\documentclass[12pt,a4paper]{article}
\usepackage[utf8]{inputenc}
\usepackage{amsmath,amsthm}
\usepackage{amsfonts}
\usepackage{amssymb}
\usepackage{tikz}
\usepackage{tikz-cd}
\usepackage{caption}
\usepackage[nottoc]{tocbibind}
\usepackage[toc,page]{appendix}
\usepackage{listings}
\usepackage{pdfpages}
\usepackage{mathtools}
\usepackage[all,cmtip]{xy}
\usepackage{lipsum}
\usepackage{breqn}
\usepackage{subcaption}
\usepackage{hyperref}
\usepackage{tkz-euclide}
\usetikzlibrary{angles,quotes}
\usetikzlibrary{decorations.markings}
\usetikzlibrary{hobby}

\usepackage{authblk}

\newtheorem{corollary}{Corollary}

\newtheorem{theorem}{Theorem}

\newtheorem{lemma}{Lemma}

\numberwithin{lemma}{section}
\numberwithin{claim}{section}
\numberwithin{definition}{section}
\numberwithin{proposition}{section}
\numberwithin{equation}{section}

\def\N{\mathbb{N}}
\def\Z{\mathbb{Z}}
\def\C{\mathbb{C}}
\def\R{\mathbb{R}}

\def\<={\leq}
\def\>={\geq}
\def\inv{^{-1}}

\def\bd{\partial}

\def\l{\ell}

\newcommand{\set}[1]{\left\lbrace#1\right\rbrace}

\title{Billiard tables with rotational symmetry}

\author[1]{Misha Bialy\thanks{School of Mathematical Sciences,
    Tel Aviv University, Tel Aviv 69978, Israel.
    \textit{E-mail}: \href{mailto:bialy@tauex.tau.ac.il}{\texttt{bialy@tauex.tau.ac.il}}}\hspace{2mm}  and Daniel Tsodikovich\thanks{   School of Mathematical Sciences,
   Tel Aviv University, Tel Aviv 69978, Israel.
    \textit{E-mail}: \href{mailto:tsodikovich@tauex.tau.ac.il}{\texttt{tsodikovich@tauex.tau.ac.il}}}}

\date{}
\begin{document}
\maketitle
\begin{abstract}
We generalize the following simple geometric fact: the only centrally symmetric convex curve of constant width is a circle.
 Billiard interpretation of the condition of constant width  reads: 
   a planar curve has constant width, if and only if, the Birkhoff billiard map  inside the planar curve has a rotational invariant curve of $2$-periodic orbits.
   We generalize this statement to curves that are invariant under a rotation by angle $\frac{2\pi}{k}$,  for which the billiard map has a rotational invariant curve of $k$-periodic orbits.
    Similar result holds true also for Outer billiards and Symplectic billiards. 
    Finally, we consider  Minkowski billiards inside a unit disc of Minkowski (not necessarily symmetric) norm which is invariant under a linear map of order $k\ge 3$.
    We find a criterion for the existence of an invariant curve of $k$-periodic orbits.
     As an application, we get rigidity results for all those billiards.
\end{abstract}

\section{Introduction and Main Results}\label{section intro}
\subsection{Background} \label{subsect bg}
Mathematical billiards are fascinating examples of conservative dynamical systems. 
 Billiard systems, in many cases, provide examples of twist symplectic maps of cylinders.
Many important results are known for twist maps and will be used heavily in this paper, including Birkhoff's theorem on rotational invariant curves and Aubry-Mather variational theory (see \cite{AST_1986__144__1_0},\cite{Bangert1988MatherSF},\cite{gole2001symplectic} for relevant background). 
However, billiards form a much more rigid subclass of twist maps, where the geometry of the table plays a crucial additional role.
In this Subsection we will recall the definitions of exact twist maps, and Birkhoff, Outer, Symplectic and Minkowski billiard systems, and in particular how they are realized as exact twist maps of cylinders.
\begin{enumerate}
\item \textbf{Exact Twist Map of cylinders}: Let $\mathcal{A}=S^1\times\R$ be a cylinder. 
An area preserving diffeomorphism $T:\mathcal{A}\to\mathcal{A}, T(q,p)=(Q(q,p),P(q,p))$ is called a twist map if $\frac{\partial Q}{\partial p}$ never vanishes. 
The fact that $T$ is area preserving implies that the 1-form $PdQ - pdq$ is closed.
The map $T$ is called an exact twist map if this one form is exact.
In this case, there exists a smooth function $S$, called generating function, such that $dS=PdQ-pdq$.
The twist condition implies that function $S$ can be considered as a function of the variables $(q,Q)$ and the mixed second order partial derivative of $S$ does not vanish.
Conversely, for such $S$, the relation $PdQ-pdq=dS$ defines a twist map the generating function of which is $S$. It is important that the sequences $\{q_n\}$ corresponding to the orbits of the twist map, $(q_n,p_n):=T^n(q_0,p_0)$, are extremals of the variational principle $\sum_{n}S(q_n, q_{n+1})$.
\item \textbf{Birkhoff Billiards}: This model was introduced and studied extensively by G.D.Birkhoff.
A point mass moves in a straight line inside a convex planar domain with smooth boundary $\gamma$. 
When hitting the boundary, it reflects according to the classical law of geometric optics: the angle of incidence is equal to the angle of return.
The phase cylinder of the Birkhoff billiard system is the space of oriented lines intersecting $\gamma$, and it can be parametrized the following way: if $s$ is the arc length parameter on $\gamma$, then every oriented line that intersects $\gamma$ is uniquely determined by the pair $(s,\alpha)$, where $s$ is the arc length parameter of the first intersection point of the line with $\gamma$, and $\alpha$ is the angle between the line and $\dot{\gamma}(s)$. 
The Birkhoff billiard map is the twist map that corresponds to the generating function $L(s,s')=|\gamma(s)-\gamma(s')|$.
See Figure \ref{fig: birkBilliard}.
   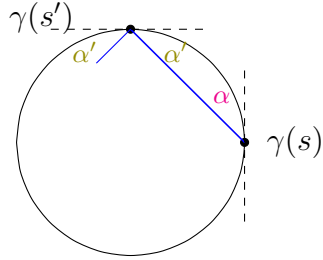
\begin{figure}
        \centering
       \begin{tikzpicture}[scale = 1.5]
   \draw[domain = -180:180, smooth, variable = \t, black] plot({cos(\t)},{sin(\t)});
   \tkzDefPoint(1,0){A};
   \tkzDefPoint(0,1){B};
   \tkzDrawPoints(A,B);
   \tkzDrawSegment[blue](A,B);
   \draw[dashed](1,-0.7)--(1,0.7);
   \draw[dashed](-0.7,1)--(0.7,1);
   \draw[blue](0,1)--(-0.3,0.7);
   \node[right] at (1.1,0) {$\gamma(s)$};
   \node[left] at (-0.4,1.1) {$\gamma(s')$};
   \node[magenta] at (0.8, 0.4) {\footnotesize{$\alpha$}};
   \node[olive] at (0.4,0.8) {\footnotesize{$\alpha'$}};
   \node[olive] at (-0.4,0.8) {\footnotesize{$\alpha'$}};
       \end{tikzpicture}
       \caption{Birkhoff billiard map.\label{fig: birkBilliard}}
    \end{figure}
\item \textbf{Outer Billiards}: Outer billiards were introduced by B. Neumann, and popularized by Moser as a toy model for planetary motion. Unlike Birkhoff billiards, the Outer billiard map acts on the exterior of a closed convex curve.
It is defined as follows (see, e.g., \cite{tabachnikov2005geometry}, \cite{BoylandOuterBilliardImpactOscilators}): given a smooth strictly convex curve $\gamma$ and a point $x$ in the exterior of $\gamma$, consider the two tangents from $x$ to $\gamma$.
Pick the right tangency point from the point of view of $x$, and reflect $x$ through this tangency point. 
Let $t$ be any parameter of $\gamma$.
If $x$ is a point in the exterior of $\gamma$ then there are unique $t,t_1$, and $\lambda,\lambda_1>0$ for which $x=\gamma(t)+\lambda\dot{\gamma}(t)=\gamma(t_1)-\lambda_1\dot{\gamma}(t_1)$. 
The outer billiard map acts on the (infinite) cylinder that has the coordinates $(t,\lambda)$, by $T(t,\lambda)=(t_1,\lambda_1)$.
See Figure \ref{fig OuterBilliardReflectionAsterisk}.
The outer billiard map is a twist map of this cylinder, that corresponds to the following generating function:
\[S(t,t_1)=\mathrm{Area}(\mathrm{conv}(\gamma\cup\set{x})) \,,\]
where $x$ is the intersection point of the tangents to $\gamma$ at $t,t_1$.
\begin{figure}
\begin{center}
\begin{tikzpicture}[scale = 2]
\draw[domain = -60:100, smooth, variable = \t, black] plot({cos(\t)},{0.3*cos(\t)+sin(\t)});
\tkzDefPoint(0.70710,-0.49497){A};
\tkzDefPoint(0.70710,0.91923){B};
\tkzDrawPoints(A,B);
\tkzDefPoint(1.41421,0.42426){C};
\tkzDrawPoint(C);
\tkzDrawSegment[dashed](A,C);
\tkzDrawSegment(B,C);
\tkzDefPoint(0,1.41421){D};
\tkzDrawSegment[blue](C,D);
\tkzDrawPoint(D);
\node[left] at (0.6,-0.49) {$\gamma(t)$};
\node[right] at (1,-0.1) {$\lambda$};
\node[right,magenta] at (1.414,0.424) {$(t,\lambda)$};
\node[below] at (0.65,0.9) {$\gamma(t_1)$};
\node[above,olive] at (1.1,0.65) {$\lambda_1$};
\node[right,olive] at (0.4,1.2) {$\lambda_1$};
\node[above,magenta] at (0,1.45) {$(t_1,\lambda_1)=T(t,\lambda)$};
\end{tikzpicture}
\end{center}
\caption{Outer billiard reflection law.}\label{fig OuterBilliardReflectionAsterisk}
\end{figure}
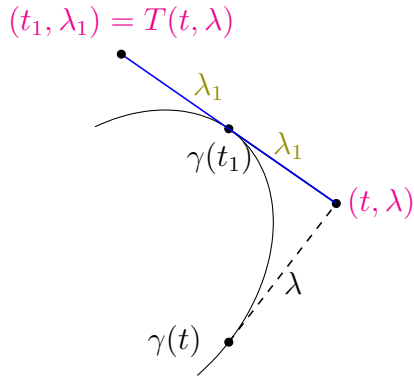
\item \textbf{Symplectic billiards}: Symplectic billiards were introduced by Albers and Tabachnikov, see \cite{albers2017introducing}.
Geometrically, this system is defined as follows: three points on a smooth convex curve $\gamma$ are three consecutive points of a symplectic billiard orbit if and only if the tangent to $\gamma$ at the second point is parallel to the segment connecting the first and the third points. 
See Figure \ref{fig SympBilliardReflectionAsterisk}.
    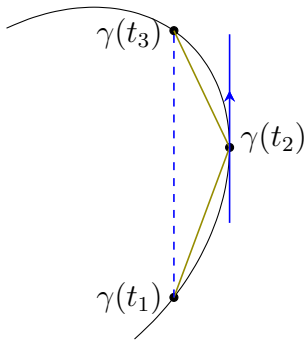
\begin{figure}
  \begin{center}
  \begin{tikzpicture}[scale = 2.5]
  \tikzset{
arr/.style={postaction=decorate,
decoration={markings,
mark=at position .7 with {\arrow[thick]{#1}}
}}}
  \draw[domain = -60:100, smooth, variable = \t, black] plot({cos(\t)},{0.3*cos(\t)+sin(\t)});
\tkzDefPoint(0.70710,-0.49497){A};
\tkzDefPoint(0.70710,0.91923){B};
\tkzDefPoint(1,0.3){C};
\tkzDrawPoints(A,B,C);
\tkzDrawSegment[olive](A,C);
\tkzDrawSegment[olive](B,C);
\tkzDrawSegment[dashed,blue](A,B);
\tkzDefPoint(1,-0.1){D};
\tkzDefPoint(1,0.9){E};
\tkzDrawSegment[arr=stealth,blue](D,E);
\node[left] at (0.707,-0.5) {$\gamma(t_1)$};
\node[left] at (0.707, 0.9) {$\gamma(t_3)$};
\node[right] at (1,0.35) {$\gamma(t_2)$};
  \end{tikzpicture}
  \caption{Symplectic billiards reflection law.}\label{fig SympBilliardReflectionAsterisk}
  \end{center}
  \end{figure}
This map can also be described as a twist map in the following way: given any parametrization $\gamma(t)$ of $\gamma$, the symplectic billiard map is the one that corresponds to the generating function $S(t,t_1)=[\gamma(t),\gamma(t_1)]$, where $[\cdot,\cdot]$ denotes the determinant of two vectors in $\R^2$. 
\item \textbf{Minkowski Billiards}:  Minkowski  and, more generally, Finsler billiards were introduced and studied by Gutkin and Tabachnikov in \cite{GUTKIN2002277}.
 Our motivation comes from the recent work by Artstein-Avidan, Karasev, and Ostrover \cite{Artstein_Avidan_2014} on Mahler's conjecture.
 Consider a smooth convex curve $\gamma$, and a Minkowski (not necessarily symmetric) norm $N$.
 The Minkowski billiard system in $\gamma$ with respect to the norm $N$ is defined as a twist map that corresponds to the generating function \[L(t,t_1)=N(\gamma(t_1)-\gamma(t))\,,\]
 where $\gamma$ is parametrized by $N$-arc length parameter (meaning, $N(\dot{\gamma}(t))=1$).
Thus, one sees that the Birkhoff billiard system is a special case of the Minkowski billiard system where one chooses $N$ to be the Euclidean norm.
In this paper we will work with a specific Minkowski norm, which is the norm determined by $\gamma$ itself. 
If $K$ is the interior of $\gamma$, then we will choose $N$ to be the gauge function of $K$:
\[N(x)=g_K(x)=\inf\set{r>0\mid x\in rK} \,.\] 
\end{enumerate}
\subsection{Formulation of The Results} \label{subsect Formulation}
Our motivation comes from recent progress towards Birkhoff conjecture for Birkhoff billiards, see for example \cite{KaloshinSorrentinoOnIntegrability}, and especially \cite{bialy2020birkhoffporitsky}, where Birkhoff conjecture was proved for centrally symmetric curves. In this paper we analyze billiard tables that have higher order symmetry.
Namely, we consider billiard tables with rotational symmetry of order $k\ge 3$.
We show that among such tables, the only ones that have a rotational invariant curve that consists of $k$-periodic orbits are circles (an exact formulation will appear below).
Note that this claim does not assume foliation by invariant curves, but assumes the existence of a specific invariant curve, the one consisting of $k$-periodic orbits.
This result is a generalization of the following simple geometric fact:
a centrally symmetric planar convex domain of constant width must be a disc.
This is a generalization because a centrally symmetric convex domain is just one that is invariant under a rotation of order two, and a convex domain of constant width is one in which the billiard ball map has a rotational invariant curve of 2-periodic orbits.
Moreover, we will study a similar setting for other billiard systems - Outer billiards (see, e.g., \cite{BoylandOuterBilliardImpactOscilators},\cite{tabachnikov2005geometry}), Symplectic billiards (see, e.g., \cite{albers2017introducing}), and Minkowski billiards (see, e.g., \cite{GUTKIN2002277}).
In the first two cases we will obtain analogous results, but in the last case we will obtain a slightly different result.
Since these systems are invariant under all affine transformations, we assume in these settings that the billiard table is invariant under a (finite) order $k\ge 3$ element of $\textrm{GL}(2,\R)$.
In the first two cases, we show that if for such a domain there exists a rotational invariant curve of $k$-periodic orbits, then it must be an ellipse, and in the last case we find a complete characterization of the domains for which such rotational invariant curve exists.
It is an elementary fact in linear algebra that an order $k\ge 3$ element in $\textrm{GL}(2,\R)$ is conjugate to a rotation by angle $\frac{2\pi m}{k}$, where $m$ and $k$ are coprime.
For such $k,m$, any set that is invariant under a rotation by $\frac{2\pi m}{k}$ is also invariant under a rotation by $\frac{2\pi}{k}$.
For this reason, we can always assume that the order $k$ linear map is conjugate to a rotation by angle $\frac{2\pi}{k}$. 
Here and below, we write $R_\theta$ for the rotation of $\R^2$ by angle $\theta$.

Our main results are the following:
\begin{theorem}\label{theorem euclideanK}
Let $\gamma$ be a $C^2$-smooth, planar, strictly convex curve which is invariant under $R_{\frac{2\pi}{k}}$, with $k\ge 3$.
 If the Birkhoff billiard map inside $\gamma$ has a rotational invariant curve of $k$-periodic orbits, then $\gamma$ is a circle.
\end{theorem}
\begin{theorem}\label{theorem outer}
Let $\gamma$ be a $C^2$-smooth, planar, strictly convex curve which is invariant under a linear map of order $k\ge 3$.
 If the Outer billiard map in the exterior of $\gamma$ has a rotational invariant curve of $k$-periodic orbits, then $\gamma$ is an ellipse.
\end{theorem}
Here the restriction $k\ge 3$ is natural: there are no $2$-periodic orbits for the Outer billiard map.

Now we turn to an important application of Theorem \ref{theorem outer} to the	 geometry of Radon planes. 
A normed plane $( \mathbb {R}^2, ||\cdot||)$ is called a Radon plane if the Birkhoff orthogonality relation is symmetric (see \cite{MartiniHorst2006AaRc}).
The unit circle of a Radon norm is called a Radon curve.
An equivalent definition in terms of the Outer billiard map was suggested by S. Tabachnikov (\cite {tabachnikovPC}): 
{\it Let $\gamma$ be the unit circle of the norm, then the norm is Radon if and only if the outer billiard map corresponding to $\gamma$ has a rotational invariant curve of 4-periodic orbits. }Using the last definition we get an immediate corollary of Theorem  \ref{theorem outer}:
\begin{corollary}
	If a Radon norm on $\mathbb {R}^2$ is invariant under a linear map of order $4$ then the norm is necessarily Euclidean.
	\end{corollary}
Remarkably, examples of non-circular analytic Radon curves having symmetries of 
order $2l$, with odd $l\geq3$, were constructed in a recent paper, see  \cite{bialy2020selfbacklund}, and Figure \ref{fig radonExamples}.
	\begin{figure}
	\begin{center}
	
	\includegraphics[width=1.0\textwidth]{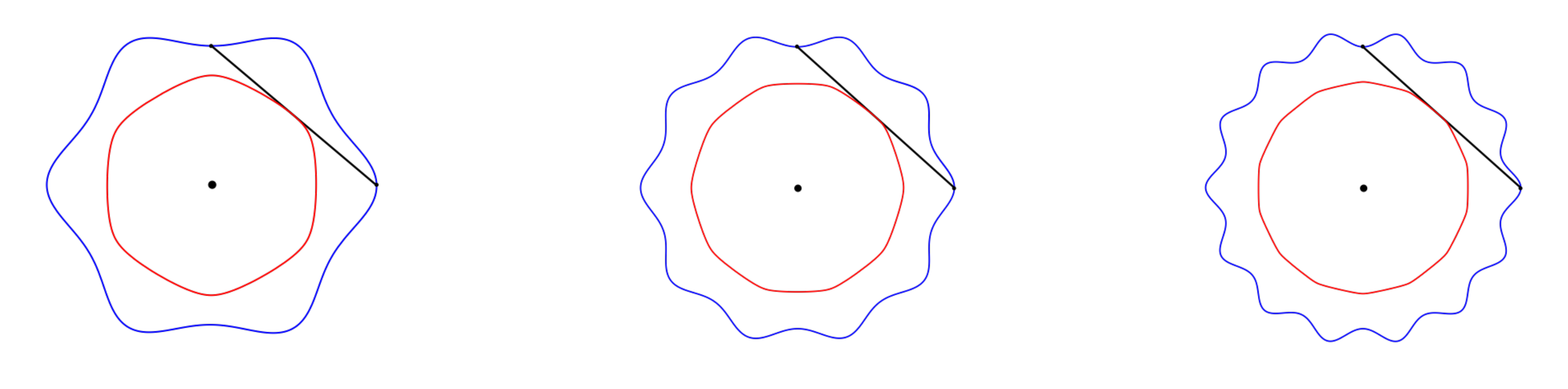}
	\caption{Radon curves (red) and invariant curve of 4-periodic orbits (blue) having symmetries of order $6,10,14$ respectively}\label{fig radonExamples}
	
	\end{center}
\end{figure}

Now we state our result for Symplectic and Minkowski billiards.
\begin{theorem}\label{theorem symplectic}
Let $\gamma$ be a $C^2$-smooth, planar, strictly convex curve which is invariant under a linear map of order $k\ge3$.
 If the Symplectic billiard map inside $\gamma$ has a rotational invariant curve of $k$-periodic orbits, then $\gamma$ is an ellipse.
\end{theorem}
In this theorem, the assumption that $k\ge 3$ is necessary: by definition of Symplectic billiards, every curve $\gamma$ admits an invariant curve of two periodic orbits: these are the orbits between the pairs of points of $\gamma$ that have parallel tangents. 

\begin{theorem}\label{theorem Minkowski}
Let $\gamma$ be a $C^2$-smooth, planar, strictly convex curve, and $K$ be its interior. Assume that $K$ is invariant under a linear map $A$ of order $k\ge 3$.
 If the Minkowski billiard map in $K$, with the norm induced by $K$, has a rotational invariant curve of $k$-periodic orbits, then $K$ is invariant under a linear map  of order $ak$, which commutes with $A$,  where $a$ depends on the remainder of $k$ modulo $4$:
\[\begin{cases}
a = 1\,, \textrm{if } k\equiv 2 \pmod{4} \,,\\
a = 2\,, \textrm{if } k\equiv 0 \pmod{4} \,,\\
a = 4\,, \textrm{if } k\equiv 1 \pmod{2} \,.\\
\end{cases}\]
Moreover, these conditions are sufficient for the existence of rotational invariant curves of $k$-periodic orbits, of all rotation numbers $\frac{r}{k}$, where $r$ is coprime with $k$. 
\end{theorem}
We remark that this theorem also holds for $k=2$, that is for centrally symmetric $K$.
Indeed, for centrally symmetric $K$, the chords connecting $x\in\bd K$ with $-x$ are Minkowski billiard orbits with respect to the norm induced by $K$.

We illustrate Theorem \ref{theorem Minkowski} with the following examples.
	 
	Let $\gamma=\partial  K$ be a small perturbation of a regular polygon. This perturbation is chosen in such a way that $\gamma$ is smooth, strictly convex, and has the same symmetries as the original regular polygon.
	We shall use the fact that all $k$-periodic orbits corresponding to the same rotational invariant curve have the same action and hence the same (Minkowski) perimeter.   In these examples, $i$ and $i'$,  
	$i=1,2,3...$ denote the intersection points of $\gamma $ with the symmetry axes, see Figure \ref{fig: exampleToThm4}.

	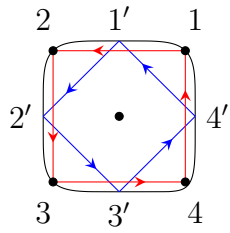
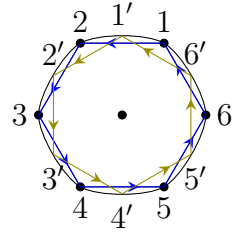
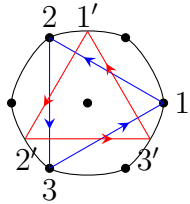
\begin{figure}
		\tikzset{
			arr/.style={postaction=decorate,
				decoration={markings,
					mark=at position .7 with {\arrow[thick]{#1}}
		}}}
		\centering
		\begin{subfigure}[b]{0.4\textwidth}
			\centering
			\begin{tikzpicture}
			\draw[domain = 0:90, smooth, variable = \t, black] plot({cos(\t)^(0.4)},{sin(\t)^(0.4)});
			\draw[domain = 90:180, smooth, variable = \t, black] plot({-abs(cos(\t))^(0.4)},{sin(\t)^(0.4)});
			\draw[domain = 180:270, smooth, variable = \t, black] plot({-abs(cos(\t))^(0.4)},{-abs(sin(\t))^(0.4)});
			\draw[domain = 270:360, smooth, variable = \t, black] plot({abs(cos(\t))^(0.2)},{-abs(sin(\t))^(0.4)});
			\draw[arr=stealth,blue] (1,0)--(0,1);
			\draw[arr=stealth,blue] (0,1)--(-1,0);
			\draw[arr=stealth,blue] (-1,0)--(0,-1);
			\draw[arr=stealth,blue] (0,-1)--(1,0);
			\draw[arr=stealth,red] (0.87,0.87)--(-0.87,0.87);
			\draw[arr=stealth,red] (-0.87,0.87)--(-0.87,-0.87);
			\draw[arr=stealth,red] (-0.87,-0.87)--(0.87,-0.87);
			\draw[arr=stealth,red] (0.87,-0.87)--(0.87,0.87);
			\tkzDefPoint(0.87,0.87){A};
			\tkzDefPoint(0.87,-0.87){B};
			\tkzDefPoint(-0.87,-0.87){C};
			\tkzDefPoint(-0.87,0.87){D};
			\tkzDefPoint(0,0){O};
			\tkzDrawPoints(A,B,C,D,O);
			\node[above] at (1,1) {$1$};
			\node[above] at (-1,1) {$2$};
			\node[below] at (1,-1) {$4$};
			\node[below] at (-1,-1) {$3$};
			
			\node[above] at (0,1) {$1'$};
			\node[left] at (-1,0) {$2'$};
			\node[below] at (0,-1) {$3'$};
			\node[right] at (1,0) {$4'$};
			
			\end{tikzpicture}
			\caption{Two four periodic orbits in a square of different Minkowski perimeters.}\label{subfig orbInSquare}
		\end{subfigure}
		\hspace{1em}
		\begin{subfigure}[b]{0.4\textwidth}
			\centering
			\begin{tikzpicture}[scale = 1.1]
			\draw (1,0)--(0.5,0.866)--(-0.5,0.866)--(-1,0)--(-0.5,-0.866)--(0.5,-0.866)--(1,0);
			\tkzDefPoint(1,0){A};
			\tkzDefPoint(0.5,0.866){B};
			\tkzDefPoint(-0.5,0.866){C};
			\tkzDefPoint(-1,0){D};
			\tkzDefPoint(-0.5,-0.866){E};
			\tkzDefPoint(0.5,-0.866){F};
			\tkzDefPoint(0,0){O}
			\tkzDrawPoints(A,B,C,D,E,F,O);

			\draw[smooth cycle,tension = 3] plot coordinates {(1,0) (0.5,0.866) (-0.5,0.866) (-1,0) (-0.5,-0.866) (0.5,-0.866)};

			\draw[arr=stealth,blue](1,0)--(0.5,0.866);
			\draw[arr=stealth,blue](0.5,0.866)--(-0.5,0.866);
			\draw[arr=stealth,blue](-0.5,0.866)--(-1,0);
			\draw[arr=stealth,blue](-1,0)--(-0.5,-0.866);
			\draw[arr=stealth,blue](-0.5,-0.866)--(0.5,-0.866);
			\draw[arr=stealth,blue](0.5,-0.866)--(1,0);
			
			\draw[arr=stealth,olive](0,0.95)--(-0.75*1.096,0.433*1.096);
			\draw[arr=stealth,olive](-0.75*1.096,0.433*1.096)--(-0.75*1.096,-0.433*1.096);
			\draw[arr=stealth,olive](-0.75*1.096,-0.433*1.096)--(0,-0.866*1.096);
			\draw[arr=stealth,olive](0,-0.866*1.096)--(0.75*1.096,-0.433*1.096);
			\draw[arr=stealth,olive](0.75*1.096,-0.433*1.096)--(0.75*1.096,0.433*1.096);
			\draw[arr=stealth,olive](0.75*1.096,0.433*1.096)--(0,0.866*1.096);
			
			\node[above] at (0.5,0.866) {$1$};
			\node[above] at (-0.5,0.866) {$2$};
			\node[left] at (-1,0) {$3$};
			\node[below] at (-0.5,-0.866) {$4$};
			\node[below] at (0.5,-0.866) {$5$};
			\node[right] at (1,0) {$6$};
			
			\node[above] at (0,0.866*1.096) {$1'$};
			\node[above] at (-0.75*1.096,0.433*1.096) {$2'$};
			\node[below] at (-0.75*1.096,-0.433*1.096) {$3'$};
			\node[below] at (0,-0.866*1.096) {$4'$};
			\node[below] at (0.8*1.096,-0.44*1.096) {$5'$};
			\node[above] at (0.8*1.096,0.44*1.096) {$6'$};
			\end{tikzpicture}
			\caption{Two six periodic orbits in a hexagon with the same Minkowski perimeter.}\label{subfig sameSixInHexagon}
		\end{subfigure}
		\begin{subfigure}[b]{0.4\textwidth}
			\centering
			\begin{tikzpicture}
			
			\tkzDefPoint(1,0){A};
			\tkzDefPoint(0.5,0.866){B};
			\tkzDefPoint(-0.5,0.866){C};
			\tkzDefPoint(-1,0){D};
			\tkzDefPoint(-0.5,-0.866){E};
			\tkzDefPoint(0.5,-0.866){F};
			\tkzDefPoint(0,0){O}
			\tkzDrawPoints(A,B,C,D,E,F,O);
			
			\draw[smooth cycle,tension = 3] plot coordinates {(1,0) (0.5,0.866) (-0.5,0.866) (-1,0) (-0.5,-0.866) (0.5,-0.866)};
			
			\draw[arr=stealth,blue] (1,0)--(-0.5,0.866);
			\draw[arr=stealth,blue] (-0.5,0.866)--(-0.5,-0.866);
			\draw[arr=stealth,blue] (-0.5,-0.866)--(1,0);
			
			\node[right] at (1,0) {$1$};
			\node[above] at (-0.5,0.866){$2$};
			\node[below] at (-0.5,-0.866){$3$};
			
			\draw[arr=stealth,red] (0,0.866*1.096)--(-0.75*1.096,-0.433*1.096);
			\draw[arr=stealth,red] (-0.75*1.096,-0.433*1.096)--(0.75*1.096,-0.433*1.096);
			\draw[arr=stealth,red] (0.75*1.096,-0.433*1.096)--(0,0.866*1.096);
			
			\node[above] at (0,0.866) {$1'$};
			\node[below] at (-0.8,-0.44) {$2'$};
			\node[below] at (0.8,-0.44) {$3'$};
			
			\end{tikzpicture}
			\caption{Two three periodic orbits in a hexagon with different Minkowski perimeters.}\label{subfig diffThreeInHexagon}
		\end{subfigure}
		\caption{Examples to Theorem \ref{theorem Minkowski}.}\label{fig: exampleToThm4}
	\end{figure}
{\it
\begin{enumerate}
\item Let $\gamma=\partial K$ be perturbation of the boundary of a square. 
	$K$ has a symmetry of order $4$, but not of order $8$, so by Theorem \ref{theorem Minkowski}, the Minkowski billiard map inside $K$ has no  invariant curves of $4$ periodic orbits. 
	It is easy to see that the orbits $1234$ and $1'2'3'4'$ in Figure \ref{subfig orbInSquare} are  Minkowski billiard orbits. 
	The Minkowski perimeter of the first orbit is (approximately) $8$, while the perimeter of the second orbit is (approximately) $4$.
	Since these perimeters are not equal, we conclude that there is no invariant curve of $4$-periodic orbits. 
	\item As another example, let $K$ be the perturbation of a regular hexagon. 
	Here again we have two orbits, $123456$ and $1'2'3'4'5'6'$, see Figure \ref{subfig sameSixInHexagon}. Both of them are Minkowski billiard orbits. 
	Their Minkowski perimeters are equal (both equal $6$),  which is consistent with the fact that in this case an invariant curve of $6$-periodic orbits does exist by Theorem \ref{theorem Minkowski}.
	\item On the other hand, if we consider the 3-periodic orbits,  $123$ and $1'2'3'$ in Figure \ref{subfig diffThreeInHexagon}, then these are also Minkowski billiard orbits, but their Minkowski perimeters are approximately $6$ and $4.5$ respectively, which means that no invariant curve of $3$-periodic orbits exists, in agreement with Theorem \ref{theorem Minkowski}.
\end{enumerate}
	}

Next we show that Theorem \ref{theorem Minkowski}  
 implies that for Minkowski billiards we can have coexistence of invariant curves with arbitrary rational rotation numbers.
For example, for any $3\le k_1,k_2\in\N$, let $a_{k_1},a_{k_2}$ be as in Theorem \ref{theorem Minkowski}.
Any curve $\gamma$ which is invariant under an element of $\mathrm{GL}(2,\R)$ of order $\mathrm{lcm}(a_{k_1}k_1,a_{k_2}k_2)$ is also invariant under an element of order $a_{k_1}k_1$ and under an element of order $a_{k_2}k_2$.
As a result, Theorem \ref{theorem Minkowski} implies the existence of rotational invariant curves for the Minkowski billiard map in $\gamma$  of rotation numbers $\frac{r_1}{k_1}$ and $\frac{r_2}{k_2}$, where $r_1,r_2$ are coprime with $k_1,k_2$ respectively.
This construction generalizes naturally to any finite set of rational numbers.
Interestingly, an analogous construction for Birkhoff billiards is not known. 

The proofs of our main results contain very similar ingredients: a convenient parametrization of $\gamma$, and Aubry-Mather theory allow us to deduce that the bounce points of the $k$-periodic orbits coincide with orbits of the symmetry of the table.
In the first three types of billiards, we use this to derive trigonometric equations that cannot be solved by rational multiples of $\pi$.
However, if $\gamma$ was not a circle (for the Birkhoff case) or an ellipse (for the Outer and Symplectic cases), we find a rational multiple of $\pi$ that solves these equations, getting a contradiction.
In fact, the last step in our proof of Theorem \ref{theorem euclideanK} can be replaced as follows: instead of considering trigonometric equations, one can use the Lemma proved in \cite{InnamiNobuhiro1988Ccwp}.

In the Minkowski case, we see that the fact that  the points of the $k$-periodic orbit coincide with the orbit of the symmetry, yields the additional symmetry of $K$ guaranteed by the theorem.

\subsection{Total Integrability}\label{subsect totalInteg}
In this section we apply our main results to the question of total integrability.
Let $T$ be an area preserving twist map of the cylinder, and $U$ be an open subset of the cylinder.
We say that $T$ is \textit{totally integrable in $U$}, if $U$ admits a $C^1$ foliation by rotational invariant curves of $T$.
The following corollaries follow immediately from the main theorems formulated in the previous subsection:

\begin{corollary}\label{cor notTotaIntegrableBirkhoff}
Let $3\<= k\in\N$, and $\gamma$ be a $C^2$-smooth, planar, strictly convex curve which is invariant under $R_{\frac{2\pi}{k}}$. 
Suppose that the Birkhoff billiard map inside $\gamma$ admits a rotational invariant curve of rotation number $\frac{r}{k}$, with $r$ coprime to $k$.
Suppose that this curve has a neighborhood $U$ in which the Birkhoff billiard map is totally integrable.
Then $\gamma$ is a circle.
\end{corollary}
It was shown in \cite{Bialy1993} that the only Birkhoff billiard table for which the entire cylinder is totally integrable is a circle.
Here the symmetry assumption allows us to weaken the condition of total integrability to  a part of the phase space.

The following corollary is valid for both Outer and Symplectic billiards:
\begin{corollary}\label{cor notTotalintegrableOuterSymplectic}
Let $3\<= k\in\N$, and $\gamma$ be a $C^2$-smooth, planar, strictly convex curve which is invariant under an element of $\mathrm{GL}(2,\R)$ of order $k$.
Suppose that  the Outer  (or Symplectic) billiard map of $\gamma$ admits a rotational invariant curve of rotation number $\frac{r}{k}$, with $r$ coprime to $k$.
If this curve has a neighborhood $U$, such that the Outer (or Symplectic) billiard map is totally integrable in $U$,
then $\gamma$ is an ellipse.
\end{corollary}
\begin{corollary}\label{cor notTotalIntegrableMinkowski}
Let $3\le k\in\N$, and let $\gamma$ be a $C^2$-smooth, planar strictly convex curve which is invariant under a linear map of order $k$.
Suppose that the Minkowski billiard map in $K$, where $K$ is the interior of $\gamma$, with the norm induced by $K$, has  a rotational invariant curve of rotation number $\frac{r}{k}$, with $r$ coprime to $k$.
If this invariant curve has a neighborhood $U$ in which the Minkowski billiard map is totally integrable,
then $K$ is also invariant under a linear map of order $ak$, where $a$ is as in Theorem \ref{theorem Minkowski}.
\end{corollary}
These corollaries follow from the corresponding theorems by the following argument:
in each corollary, the invariant curve of rotation number $\frac{r}{k}$ is a leaf of a $C^1$ foliation of the open set $U$.
Therefore, this leaf inherits an absolutely continuous invariant measure.
Moreover, since the rotation number is rational, all points on this invariant curve must be periodic.
Consequently, we can use the corresponding theorem to obtain the desired conclusion.

In the case of Minkowski billiards, we can iterate Corollary \ref{cor notTotalIntegrableMinkowski}, to obtain the following rigidity result for Minkowski billiards, under a symmetry assumption:
\begin{theorem} \label{thm BirkhoffForMinkowski}
Let $\gamma$ be a $C^2$-smooth, planar, strictly convex curve which is invariant under a linear map of order $k\ge 3$.
Consider the Minkowski billiard system in $K$, the interior of $\gamma$, with the norm induced by $K$.
If the Minkowski billiard map of $K$ is totally integrable on the entire phase cylinder, then $\gamma$ is a (Euclidean) ellipse.
\end{theorem}
\begin{proof}
Suppose that the entire phase cylinder of the Minkowski billiard map is $C^1$ foliated by invariant curves.
Then, by continuity of the rotation number, there are invariant curves of all rotation numbers  in $(0,1)$.
Consider the one with rotation number $\frac{1}{k}$.
This curve has a neighborhood in which the Minkowski billiard map is totally integrable, so by Corollary \ref{cor notTotalIntegrableMinkowski}, the curve $\gamma$ is $ak$-symmetric, where $a$ is as in Theorem \ref{theorem Minkowski}.
If $k\not\equiv 2\pmod{4}$, then the number $ak$ is divisible by $4$.
If $k\equiv 2 \pmod{4}$, then $\gamma$ is also invariant under a linear map of odd order $\frac{k}{2}\ge 3$.
Repeat the previous argument, now using the invariant curve of rotation number $\frac{2}{k}$.
This yields that $\gamma$ is invariant under a linear map of order $4\cdot\frac{k}{2}=2k$, which is also divisible by $4$.
As a result, we see that for any $k$, the assumption of the theorem implies that $\gamma$ is invariant under a linear map of order $m$, which is divisible by $4$.
Repeat the argument for the invariant curve of rotation number $\frac{1}{m}$, to get that $\gamma$ is also $2m$-symmetric, and continue this way, to see that for all $n$, the curve $\gamma$ is invariant under a linear map of order $2^n m$.
Moreover, Theorem \ref{theorem Minkowski} also claims that all of those linear maps commute.
As a result, they all are simultaneously conjugate to rotation maps by angles $\frac{2\pi r_n}{2^n m}$, where $r_n$ is coprime to $2^n m$.
Work now in the basis where all of these maps are rotations.
Since $r_n$ is coprime with $2^n m$, then $\gamma$ is also invariant under rotation by $\frac{2\pi}{2^n m}$.
 Take some point $x_0\in\bd K$.
All of its rotations by angles $\frac{2\pi}{2^n m}$ are also points on $\bd K$. 
Since $n$ can be arbitrarily large, these points will be dense in the circle or radius $|x_0|$.
The boundary of $K$ is closed, so it contains this circle.
But $K$ is a convex set, so this means that $K$ must be a disc.
We see that $\gamma$ is a circle in some basis, so it is an ellipse.
\end{proof}
 \textbf{Example:} Consider the $L^p$ balls in $\R^2$ for $p> 2$. 
 It is well known (see, e.g., \cite{LpIsometry}) that the only linear isometries of the $L^p$ norm, for $p\neq 2$ are generalized permutations (that is, products of permutation matrices with diagonal matrices with diagonal entries $\pm{1}$). 
 There exists a generalized permutation of order four, for example, rotation by $\frac{\pi}{2}$, but in $\R^2$ there are no generalized permutations of order eight.
 By Corollary \ref{cor notTotalIntegrableMinkowski}, the Minkowski billiard map inside an $L^p$ ball is not totally integrable. 

{Structure of the paper is as follows:} 
Sections \ref{section birkhoff}, \ref{section proofOuter}, \ref{section proofSymplectic}, and \ref{section proofMinkowski} are dedicated to the proofs of Theorems \ref{theorem euclideanK}, \ref{theorem outer}, \ref{theorem symplectic}, and \ref{theorem Minkowski}, respectively.
Section \ref{section followUp} presents some follow up questions that arise naturally from these results.
\subsection*{Acknowledgments}
M.B is partially supproted by ISF grant 580/20, and DFG grant MA-2565/7-1 within the Middle East Collaboration Program.
 D.T is supproted by ISF grant 580/20 and ISF grant 667/18.
  We are grateful to Yaron Ostrover for useful discussions.

\section{Birkhoff Billiards - proof of Theorem \ref{theorem euclideanK}}\label{section birkhoff}
In this section we will prove Theorem \ref{theorem euclideanK}.
 Let $\gamma$ be a $C^2$-smooth, planar, strictly convex curve, which is invariant under a rotation by angle $\frac{2\pi}{k}$, $R_{\frac{2\pi}{k}}$.
  In this section, we parametrize $\gamma$ by $\psi$, which is the angle that the outer normal to $\gamma(\psi)$ makes with the $x$ axis.
The phase cylinder of the billiard ball map is the space of oriented lines intersecting $\gamma$.
We parametrize it by $\psi$, and by $\delta$, the angle between the oriented line and the tangent to $\gamma$ at $\gamma(\psi)$ (the first intersection point of the line with $\gamma$ according to the orientation of the line), as in \cite{bialy2020birkhoffporitsky} (see Figure \ref{fig parameterPhaseSpaceBirk}).
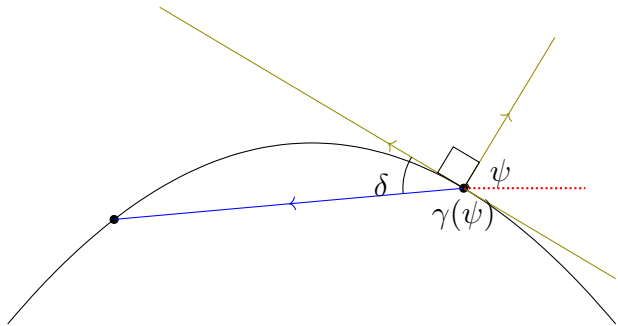
\begin{figure}
\begin{center}
\begin{tikzpicture}[scale = 2]
\draw[domain = -2:2, smooth, variable = \x, black] plot({\x},{1-0.3*\x*\x});
\tkzDefPoint(1,0.7){A};
\tkzDefPoint(-1.3,0.493){B};
\tkzDrawPoints(A,B);
\begin{scope}[decoration={
markings,
mark = at position 0.5 with {\arrow{>}}}
]
\draw[postaction = {decorate},blue](1,0.7)--(-1.3,0.493);
\draw[postaction = {decorate},olive](2,0.1)--(-1,1.9);
\draw[postaction = {decorate},olive](1,0.7)--(1.6,1.7);
\end{scope}
\tkzDefPoint(-1,1.9){C};
\tkzMarkAngle[size = 0.4, mark = none](C,A,B);
\node[left] at (0.58,0.72) {$\delta$};
\node[below] at(1,0.7) {$\gamma(\psi)$};
\draw[densely dotted,thick,red] (1,0.7) -- (1.8,0.7);
\tkzDefPoint(1.6,1.7){D}
\tkzDefPoint(1.8,0.7){E}
\tkzMarkAngle[size = 0.4, mark = none, arc = ll](E,A,D);
\node[right] at (1.1,0.8) {$\psi$};
\tkzDefPoint(2,0.1){F}
\tkzMarkRightAngle[size = 0.2](D,A,C)
\end{tikzpicture}
\caption{Parametrization of the phase space of Birkhoff billiards.}\label{fig parameterPhaseSpaceBirk}
\end{center}
\end{figure}
The billiard map is a self map of this phase cylinder, and will be denoted by $T$. 
Assume that the billiard map inside $\gamma$ has a rotational invariant curve, $\alpha$, of $k$-periodic orbits, and that the rotation number of $\alpha$ is $\frac{r}{k}$, where $0<r<\frac{k}{2}$, and $r,k$ are coprime (the case where  $\frac{k}{2}<r<k$ is obtained by reversing orientation).
Let $R$ be the map of the phase cylinder that rotates oriented lines by $\frac{2\pi r}{k}$.
The curve $\gamma$ is also invariant under $R_{\frac{2\pi r}{k}}$, so we have
\[R(\psi,\delta)=(\psi+\frac{2\pi r}{k},\delta)\,.\]
In addition, since this rotation is an isometry that preserves $\gamma$, we have
\[R\circ T = T\circ R\,.\]
By a theorem of Birkhoff (see, e.g., \cite[Chapter I]{AST_1986__144__1_0}), the invariant curve $\alpha$ is the graph of a Lipschitz function, $\delta = d(\psi)$.
We assumed that every point on this curve is periodic.
Hence, by Aubry Mather theory (see, e.g., \cite[Theorem 5.8]{Bangert1988MatherSF}), every other rotational invariant curve with the same rotation number must coincide with $\alpha$.
Since the curve $R(\alpha)$ has the same rotation number as $\alpha$, the curve $\alpha$ must be invariant under $R$.
Therefore, the function $d$ must be $\frac{2\pi r}{k}$-periodic. 
Next, consider the function 
\[S:\R\slash(2\pi\Z)\to\R\slash(2\pi\Z)\,,\]
which describes the dynamics on the invariant curve $\alpha$, in the following sense:
\[T(\psi,d(\psi))=(S(\psi),d(S(\psi)))\,.\]
Consider further its strictly increasing lift, $\tilde{S}:\R\to\R$. 
The function $\tilde{S}$ satisfies the following identity
$\tilde{S}^k(\psi)=\psi+2\pi r$.
This is because the rotation number of $\alpha$ is $\frac{r}{k}$.
Next, evaluate both sides of the equality $T\circ R = R\circ T$ on the points $(\psi,d(\psi))$:\begin{gather*}
T\circ R(\psi,d(\psi))=T(\psi+\frac{2\pi r}{k},d(\psi))=\\
=T(\psi+\frac{2\pi r}{k},d(\psi+\frac{2\pi r}{k}))=\\
=(S(\psi+\frac{2\pi r}{k}),d(S(\psi+\frac{2\pi r}{k})))\,,
\end{gather*}
and,
\begin{gather*}
R\circ T(\psi,d(\psi))=R(S(\psi),d(S(\psi)))=\\
=(S(\psi)+\frac{2\pi r}{k},d(S(\psi)))\,.
\end{gather*}
Consequently, $\tilde{S}(\psi+\frac{2\pi r}{k})=\tilde{S}(\psi)+\frac{2\pi r}{k}$. 
As the next simple argument shows, this, together with $\tilde{S}^k(\psi)=\psi+2\pi r$, implies that $\tilde{S}(\psi)=\psi+\frac{2\pi r}{k}$ for all $\psi$.
\begin{lemma}\label{lem impliesIsAShift}
Let $f:\R\to\R$ be a strictly increasing function, $L>0$, and $k\in\N$, such that:
\begin{enumerate}
\item $f(t+\frac{L}{k})=f(t)+\frac{L}{k}$, for all $t\in\R$.\label{itm additivity}
\item $f^k (t)=t+L$, for all $t\in\R$. \label{itm periodic}
\end{enumerate}

Then $f(t)=t+\frac{L}{k}$, for all $t\in\R$. 
\end{lemma}
\begin{proof}
Assume by contradiction, that for some $t\in\R$, we have, without loss of generality:
\[f(t)<t+\frac{L}{k}\,.\]
Using assumption \ref{itm additivity}, and monotonicity of $f$, we get, for $j=1,...,k$:
\[f^j(t)<f^{j-1}(t)+\frac{L}{k}\,.\]
For $j=k$, we get, by assumption \ref{itm periodic}
\[t+L<f^{k-1}(t)+\frac{L}{k}\,.\]
Sum all of these inequalities for $j=1,...,k$:
\[\sum_{j=1}^k f^j(t)<\sum_{j=1}^k(f^{j-1}(t)+\frac{L}{k})=\sum_{j=0}^{k-1}f^j(t)+L\,.\]
But left hand side is:
\[\sum_{j=1}^k f^j(t)=\sum_{j=1}^{k-1} f^j(t)+f^k(t)=L+\sum_{j=0}^{k-1} f^j(t)\,,\]
getting a contradiction.
\end{proof}

As a result, $$\tilde{S}(\psi+\frac{2\pi r}{k})=\psi+\frac{2\pi r}{k}\,,$$ and using the $\frac{2\pi r}{k}$ periodicity of $d$ already established, we get:
\[d(\psi)=d(\psi+\frac{2\pi r}{k})=d(\tilde{S}(\psi))\,.\]
So we see that the function $d$ is invariant under the restriction of the billiard map to $\alpha$.
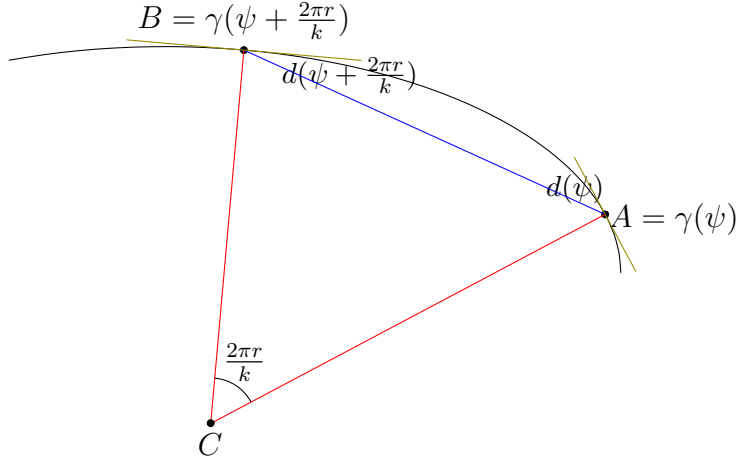
\begin{figure}
\begin{center}
\begin{tikzpicture}[scale = 1.5]
\draw[domain = 0:110, smooth, variable = \t, black] plot({4*cos(\t)},{2*sin(\t});
\draw[fill=black,black] (3.863,0.517) circle [radius = 0.03];
\draw[fill=black,black] (0.694,1.969) circle [radius = 0.03];
\node [right,black] at (3.8,0.48) {$A=\gamma(\psi)$};
\node [above,black] at (0.694,1.969) {$B=\gamma(\psi+\frac{2\pi r}{k})$};
\node [above,black] at (3.6,0.517) {\small{$d(\psi)$}};
\node [below,black] at (1.624,2.00) {\small{$d(\psi+\frac{2\pi r}{k})$}};
\draw[blue] (3.863,0.517) -- (0.694,1.969);
\draw[olive] (3.863+15*0.018,0.517-15*0.0337) -- (3.863-15*0.018,0.517+15*0.0337);
\draw[olive] (0.694 + 15*0.0687, 1.969 - 15 * 0.00606) -- (0.694 - 15* 0.0687, 1.969 + 15 * 0.00606);
\draw[red] (3.863,0.517)--(3.863-103*0.0337,0.517-103*0.018);
\draw[red] (0.694,1.969) -- (0.694 - 48*0.00606,1.969 - 48*0.0687);
\draw[fill=black,black] (0.4031,-1.3286) circle [radius = 0.03];
\node [above,black] at (0.7,-1.05) {$\frac{2\pi r}{k}$};
\node[below] at (0.4,-1.32) {$C$};
\tkzDefPoint(3.863,0.517){A};
\tkzDefPoint(0.694,1.969){B};
\tkzDefPoint(0.4031,-1.3286){C};

\tkzMarkAngle[size = 0.4, mark=none](A,C,B);
\end{tikzpicture}
\caption{Proving that $d(\psi)$ is constant, and equal to $\frac{\pi r}{k}$.}\label{fig dInvKCase}
\end{center}
\end{figure}

Finally, consider the triangle formed between the points $\gamma(\psi)$, $\gamma(\psi+\frac{2\pi r}{k})$, and the inner normals to $\gamma$ at these two points (this is triangle $\Delta ABC$ in Figure \ref{fig dInvKCase}).
 By definition of the parameter $\psi$, the angle between the two normals is $\frac{2\pi r}{k}$. 
By the definition of $d$, the angles between each normal and the chord connecting $\gamma(\psi)$ with $\gamma(\psi+\frac{2\pi r}{k})$ are $\frac{\pi}{2}-d(\psi)=\frac{\pi}{2}-d(\psi+\frac{2\pi r}{k})$. 
Comparing the sum of angles in this triangle to $\pi$, we see that $d(\psi)$ is constant and equal to $\frac{\pi r}{k}$. 
We then see that the curve $\gamma$ is a curve which satisfies the Gutkin property:
it admits a rotational invariant curve of constant angle $\frac{\pi r}{k}$. 
If $\gamma$ was not a circle, then by \cite[Corollary 7]{Gutkin_2011}, the angle $\frac{\pi r}{k}$ would solve an equation of the form:
\begin{equation}\label{gutkin}
\tan(nx)=n\tan(x).
\end{equation}

However, by \cite[Theorem 1]{cyr2011number}, this equation has no solutions among rational multiples of $\pi$.
We get a contradiction, meaning that $\gamma$ must be a circle, finishing the proof of Theorem \ref{theorem euclideanK}.

\section{Outer billiards - proof of Theorem \ref{theorem outer}} \label{section proofOuter}
In Subsection \ref{subsect bg} we recalled the twist map structure of the Outer billiard map.
Let $X$ be the exterior of $\gamma$.
 Choose an arbitrary parameter on $\gamma$, and then every point in the exterior of $\gamma$ can be uniquely written as $\gamma(t)+\lambda\dot{\gamma}(t)$ for some $t$ and $\lambda>0$. 
 The Outer billiard map can be viewed as a map $T:X\to X$, which satisfies $T(t,\lambda)=(t_1,\lambda_1)$ if and only if $\gamma(t)+\lambda\dot{\gamma}(t)=\gamma(t_1)-\lambda_1\dot{\gamma}(t_1)$, see Figure \ref{fig OuterBilliardReflectionAsterisk}.
Let $A:\R^2\to\R^2$ be the linear map as in the formulation of the theorem, which is of order $k$, meaning $A^k=\textrm{Id}=I$.
We assumed that $k\ge 3$, so $A$ is conjugate to a rotation by angle $\frac{2\pi r}{k}$, for some $r\in\N$ which is coprime to $k$, and in particular $\det A =1 $.

Suppose that for the Outer billiard map on the exterior of $\gamma$ there exists a rotational invariant curve $\alpha$ of $k$-periodic orbits, and rotation number $\frac{r}{k}$.
  The proof of Theorem \ref{theorem outer} relies on a special parametrization of $\gamma$ in which the restriction of the dynamics to $\alpha$ is given by a shift, and in which $\alpha$ is a ``horizontal" circle, $\lambda=\textrm{const}$.
  Such a parametrization is constructed in several steps.
  In fact, the first step is shared between the Outer billiards case, and the Symplectic billiards case, considered in the next section.
  Here and below, we denote by $[u,v]$ the determinant of the two vectors $u$,$v$.
  \begin{lemma}\label{lem actionOfAOnGamma}
  Let $\gamma$ be a $C^2$-smooth, planar, strictly convex curve, which is invariant under a linear map $A$ of order $k\ge 3$, and let $0<r<\frac{k}{2}$ be an integer coprime to $k$. 
 Assume that $\gamma$ is $L$-periodic, and parametrized so that
  $[\gamma(t),\dot{\gamma}(t)]=1$
  for all $t$.
   Then there exists $m\in\N$ coprime with $k$ such that:
  $B\gamma(t)=\gamma(t+\frac{Lr}{k})$ for all $t$, where $B=A^m$ is another order $k$ element of $\mathrm{GL}(2,\R)$. 
  Moreover, by rescaling $\gamma$ we may assume that $L=2\pi$. 
  \end{lemma}
  \begin{proof}
  Since $\gamma$ is invariant under $A$, for all $t$ there exists $H(t)$ such that:
  \begin{equation}\label{eq linearParameterChange}
  A\gamma(t)=\gamma(H(t))\,.
  \end{equation}
     Differentiating both sides, we get:
  \[A\dot{\gamma}(t)=H'(t)\dot{\gamma}(H(t))\,.\]
  Taking determinants with $\gamma(H(t))=A\gamma(t)$ we obtain:
  \[\Big[A\dot{\gamma}(t),A\gamma(t)\Big]=H'(t)\Big[\dot{\gamma}(H(t)),\gamma(H(t))\Big]\,.\]
  Owing to the fact that $[\gamma(t),\dot{\gamma}(t)]=1$, and properties of determinants, we get that $H'(t)=\det A = 1$, so $H$ is linear. 
  Iterating equation \eqref{eq linearParameterChange} $k$ times, gives us:
  \[\gamma(t)=A^k \gamma(t)=\gamma(H^k(t))\,.\]
  Which means that $H(t)=t+\frac{jL}{k}$, where $j\in\N$. 
  Since $A$ is conjugated to a rotation of order $k$ and $\gamma(t)\neq 0$, the points
   \[\gamma(t),A\gamma(t),...,A^{k-1}\gamma(t)\,,\]
must all be distinct.
In addition,  these points must coincide with the points $\gamma(t)$,$\gamma(t+\frac{Lj}{k})$,...,$\gamma(t+\frac{Lj(k-1)}{k})$, meaning that they are also distinct.
This implies that $j$ is coprime with $k$, and that there exists $i$ for which $\gamma(t+\frac{iLj}{k})=\gamma(t+\frac{Lr}{k})$.
Therefore, we can chose a power of $A$, for which $A^m\gamma(t)=\gamma(t+\frac{Lr}{k})$.
Then, the $k$ distinct points $\gamma(t)$,$\gamma(t+\frac{Lr}{k})$,...,$\gamma(t+\frac{Lr(k-1)}{k})$ coincide with the points $\gamma(t)$,$A^m\gamma(t)$,...,$A^{m(k-1)}\gamma(t)$, so these points are also distinct.
As a result, $m$ must be coprime with $k$, and therefore $B=A^m$ is also an order $k$ element of $\mathrm{GL}(2,\R)$.
Finally, consider the following rescaling of $\gamma$:
\[\tilde{\gamma}(t)=\sqrt{\frac{2\pi}{L}}\gamma\Big(\frac{L}{2\pi}t\Big)\,.\]
Then it is immediate to check that this curve is $2\pi$-periodic, satisfies the requirement $\Big[\tilde{\gamma}(t),\dot{\tilde{\gamma}}(t)\Big]=1$, and $B\tilde{\gamma}(t)=\tilde{\gamma}(t+\frac{2\pi r}{k})$.
  \end{proof}For convenience, we keep calling the curve with the parametrization constructed in Lemma \ref{lem actionOfAOnGamma} $\gamma$. 
 Using this parametrization, we can describe the tangency points of the $k$-periodic orbit starting at $\gamma(t)$.
  
 \begin{lemma}\label{lem tangencyPointsOuter}
 Let $\gamma$ be a $C^2$-smooth, planar, strictly convex curve. 
 Assume that $\gamma$ is invariant under a linear map $A$ of order $k$, and that the Outer billiard map on the exterior of $\gamma$ has a rotational invariant curve of $k$-periodic orbits, $\alpha$.
 Then, the tangency points of the $k$-periodic orbit starting at $\gamma(t)$ are $\gamma(t)$,$B\gamma(t)$,...,$B^{k-1}\gamma(t)$, where $B$ is as in Lemma \ref{lem actionOfAOnGamma}.
 \end{lemma}
 \begin{proof}
 Let $\frac{r}{k}$ be the rotation number of $\alpha$. 
With this $r$, parametrize $\gamma$ as in Lemma \ref{lem actionOfAOnGamma}.
  The curve $\gamma$ is invariant under $B$.
  Consider the map $R$ of the phase space, $R(t,\lambda)=(t+\frac{2\pi r}{k},\lambda)$.
  This map is induced by the action of $B$ on $X$.
 We show that  $$R\circ T = T\circ R\,.$$
 Indeed, if $T(t,\lambda)=(t_1,\lambda_1)$ then $\gamma(t)+\lambda\dot{\gamma}(t)=\gamma(t_1)-\lambda_1\dot{\gamma}(t_1)$.
  Multiplying both sides by $B$, and using the fact that $B\gamma(t)=\gamma(t+\frac{2\pi r}{k})$, and that $B\dot{\gamma}(t)=\dot{\gamma}(t+\frac{2\pi r}{k})$, we get:
 \[\gamma(t+\frac{2\pi r}{k})+\lambda\dot{\gamma}(t+\frac{2\pi r}{k})=\gamma(t_1+\frac{2\pi r}{k})-\lambda_1\dot{\gamma}(t_1+\frac{2\pi r}{k})\,.\]
 This means that $T(t+\frac{2\pi r}{k},\lambda)=(t_1+\frac{2\pi r}{k},\lambda_1)$, and this is equivalent to $T\circ R(t,\lambda) = R\circ T(t,\lambda)$.
By Birkhoff's theorem, the rotational invariant curve $\alpha$ is a graph of a function $\lambda=\lambda(t)$.
  Just like in the proof of Theorem \ref{theorem euclideanK} in Section \ref{section birkhoff}, since every point of $\alpha$ is periodic, $\alpha$ is preserved under $R$, and therefore the function $\lambda(t)$ is $\frac{2\pi r}{k}$-periodic. 
  Denote by 
  \[S:\R\slash(2\pi\Z)\to\R\slash(2\pi\Z)\]
   the function which describes the restriction of the dynamics to $\alpha$:
  \[T(t,\lambda(t))=(S(t),\lambda(S(t)))\,.\]
  Denote its lift to $\R$ by $\tilde{S}$. 
  Now evaluate the equality $T\circ R = R\circ T$ on the points $(t,\lambda(t))$ to get:
\begin{gather*}
T\circ R(t,\lambda(t))=T(t+\frac{2\pi r}{k},\lambda(t))=T(t+\frac{2\pi r}{k},\lambda(t+\frac{2\pi r}{k})) = \\
=(S(t+\frac{2\pi r}{k}),\lambda(S(t+\frac{2\pi r}{k})))\,,
\end{gather*}
and, 
  \[R\circ T(t,\lambda(t)) = R(S(t),\lambda(S(t)))=(S(t)+\frac{2\pi r}{k},\lambda(S(t)))\,.\]
  So as a result, we see that
\[  \tilde{S}(t+\frac{2\pi r}{k})=\tilde{S}(t)+\frac{2\pi r}{k}\,.\]
The function $\tilde{S}$ is increasing, and since $\alpha$ is an invariant curve of rotation number $\frac{r}{k}$, we have $\tilde{S}^k(t)=t+2\pi r$.
We can use Lemma \ref{lem impliesIsAShift} to obtain that $\tilde{S}(t)=t+\frac{2\pi r}{k}$ for all $t\in\R$.
This means that the tangency points of the $k$-periodic orbit are exactly at the points $\gamma(t)$, $\gamma(t+\frac{2\pi r}{k})$,...,$\gamma(t+\frac{2\pi r(k-1)}{k})$.
But since $\gamma(t+\frac{2\pi r}{k})=B\gamma(t)$, these points are $\gamma(t)$,$B\gamma(t)$,...,$B^{k-1}\gamma(t)$.
  \end{proof}
  Now, write the Outer billiards equation for two consecutive points of the $k$-periodic orbit:
\begin{equation}\label{eq outerBilliardEqInvCurve}
\gamma(t)+\lambda(t)\dot{\gamma}(t)=\gamma(t+\frac{2\pi r}{k})-\lambda(t+\frac{2\pi r}{k})\dot{\gamma}(t+\frac{2\pi r}{k})\,.
\end{equation}
  Since $\lambda$ is $\frac{2\pi r}{k}$-periodic, and $\gamma(t+\frac{2\pi r}{k})=B\gamma(t)$, we can write:
\begin{equation}\label{eq OuterBilliardEqInvCurveInvForm}
\lambda(t)(B+I)\dot{\gamma}(t)=(B-I)\gamma(t)\,.
\end{equation}  
Equation \eqref{eq OuterBilliardEqInvCurveInvForm} was derived from the parametrization of $\gamma$ from Lemma \ref{lem actionOfAOnGamma}, but in fact, it holds for any parametrization of $\gamma$, changing the function $\lambda(t)$ if necessary.
The next step is to find a parametrization for which the function $\lambda(t)$ will be constant.
The construction will rely on the following lemma, which will also be used in the proof of Theorem \ref{theorem symplectic}.
\begin{lemma}\label{lem ODElemma}
Let $G$ be a positive, $C^1$-smooth, $P$-periodic function on $\R$. Then there exists $a>0$ and a $C^2$-smooth increasing function $f:\R\to\R$ such that $af'(s)=G(f(s))$, and $f(s+P)=f(s)+P$ for all $s\in\R$.
\end{lemma}
\begin{proof}
For arbitrary $a>0$, consider the ordinary differential equation:
\begin{equation} \label{eq ODEgeneral}
af'(s)=G(f(s))\,.
\end{equation}
This is an autonomous equation, and $G$ and its derivative are bounded, so every solution is defined over $\R$, and the set of solutions is closed under shifting the variable: if $f(s)$ is a solution, then so is $f(s-s_0)$. 
Also, for all $\lambda>0$, if $f(s)$ is a solution, then so will be $f(\lambda s)$, but with a different constant $a$ in the equation. 
Finally, if $f(s)$ is a solution, then so is $f(s)+P$, since $G$ is $P$-periodic. 
Since $G$ is $C^1$-smooth, then every solution is $C^2$-smooth.
Now, fix an arbitrary $a$, and consider the solution to equation \eqref{eq ODEgeneral} which satisfies $f(0)=0$. 
Since $G$ is bounded from below by a positive constant, $\inf f' > 0$, which means that $\lim\limits_{s\to\infty} f(s)=\infty$. 
Therefore, we can find $c>0$ such that $f(c)=P$. 
As a result, the two functions $f(s)+P$ and $f(s+c)$ are solutions to equation \eqref{eq ODEgeneral} with the same $a$, and they coincide at $0$, so they must be equal: $f(s+c)=f(s)+P$. 
By rescaling, which might change $a$, we may assume that $c=P$. 
We therefore found a solution to equation \eqref{eq ODEgeneral} for some $a>0$, which satisfies $f(s+P)=f(s)+P$ for all $s$.
Also, since $G$ is positive, the function $f$ is increasing, as required.  
\end{proof}
Using this lemma, we will construct the required parametrization of $\gamma$:
  \begin{lemma}\label{lem goodParameterOuter}
  Let $\gamma$, $A$, $m$,$B$ be as in Lemma \ref{lem actionOfAOnGamma}.
  If the Outer billiard map in the exterior of $\gamma$ has a rotational invariant curve of $k$-periodic orbits, $\alpha$, then there exists a $C^2$-reparamterization of $\gamma$, $\mu(s)=\gamma(f(s))$ for which:
  \begin{enumerate}
  \item $\mu(s+\frac{2\pi r}{k})=B\mu(s)$ for all $s$.
 \item The function $\lambda(s)$ determined by this parametrization in equation \eqref{eq OuterBilliardEqInvCurveInvForm} is constant.
 \item $\mu$ is $2\pi$-periodic.
  \end{enumerate}
  \end{lemma}
  \begin{proof}
Let $G(t)$ be the following function (where $\gamma$ is parametrized by the parametrization of Lemma \ref{lem actionOfAOnGamma}):
\[G(t)=\frac{\Big[(B-I)\gamma(t),\dot{\gamma}(t)\Big]}{\Big[B\dot{\gamma}(t),\dot{\gamma}(t)\Big]}\,.\]
This function is well defined: the vectors $\dot{\gamma}(t)$ and $B\dot{\gamma}(t)$ are the directions of the tangents to $\gamma$ at two consecutive points of the Outer billiard orbit;
as such, these tangents must intersect, meaning that the denominator does not vanish. 
Also, the vector $(B-I)\gamma(t)$ is the vector between the first point of the orbit, and the second one. 
We therefore see that the vector pairs in both the numerator and the denominator are oriented negatively, so $G(t)>0$. 
The function $G$ is $C^1$-smooth (since $\gamma$ is $C^2$-smooth), and we show that it is $\frac{2\pi}{k}$-periodic.
Since $r$ is coprime to $k$, there exists $\l\in\N$ for which $B^\l\gamma(t)=\gamma(t+\frac{2\pi}{k})$.
As a result, when computing $G(t+\frac{2\pi}{k})$, both numerator and denominator will have an extra factor of $\det B^\l = 1$.
We conclude that the function $G$ is smooth, positive, and $\frac{2\pi}{k}$-periodic.
By Lemma \ref{lem ODElemma}, there exists $\lambda>0$ and a $C^2$-smooth increasing $f:\R\to\R$ which satisfies:
\begin{equation}\label{eq ODEOuter}
\lambda f'(s)=G(f(s))\,,
\end{equation}
and in addition $f(s+\frac{2\pi}{k})=f(s)+\frac{2\pi}{k}$ for all $s\in\R$.
This implies that for all $s\in\R$ we also have $f(s+\frac{2\pi r}{k})=f(s)+\frac{2\pi r}{k}$.
Let $\mu(s)=\gamma(f(s))$ be a $C^2$-reparametrization of~$\gamma$. 
We show that this parametrization has all the required properties. 
Indeed:
\[\mu(s+\frac{2\pi r}{k})=\gamma(f(s+\frac{2\pi r}{k}))=\gamma(f(s)+\frac{2\pi r}{k})=B\gamma(f(s))=B\mu(s)\,,\]
which proves the first item.
A similar computation with $B^\l$ instead of $B$ also gives us that $B^\l\mu(s)=\mu(s+\frac{2\pi}{k})$, and since $B^{\l k}=I$, then $\mu$ is $2\pi$-periodic, proving the last item.
We prove the second item. 
Let $\lambda(s)$ be the function for which equation \eqref{eq OuterBilliardEqInvCurveInvForm} holds:
\[\lambda(s)(B+I)\dot{\mu}(s)=(B-I)\mu(s)\,.\]
Take determinants of both sides with the vector $\dot{\mu}(s)$:
\[\lambda(s)\Big[(B+I)\dot{\mu}(s),\dot{\mu}(s)\Big] = \Big[(B-I)\mu(s),\dot{\mu}(s)\Big]\,.\]
By properties of determinant, the determinant on the left hand side is $[B\dot{\mu}(s),\dot{\mu}(s)]$ which is not zero, so we can isolate $\lambda(s)$:
\[\lambda(s)=\frac{\Big[(B-I)\mu(s),\dot{\mu}(s)\Big]}{\Big[B\dot{\mu}(s),\dot{\mu}(s)\Big]}\,.\]
Now plug here $\mu(s)=\gamma(f(s))$, $\dot{\mu}(s)=f'(s)\dot{\gamma}(f(s))$:
\[\lambda(s)=\frac{f'(s)\Big[(B-I)\gamma(f(s)),\dot{\gamma}(f(s))\Big]}{f'(s)^2\Big[B\dot{\gamma}(f(s)),\dot{\gamma}(f(s))\Big]}=\frac{G(f(s))}{f'(s)}\,.\]
But since $f$ is a solution to equation \eqref{eq ODEOuter}, right hand side is constant, proving that $\lambda(s)$ is constant, so $\mu$ has all of the required properties.
  \end{proof}
  
  For convenience, we keep calling the curve with the parametrization constructed in Lemma \ref{lem goodParameterOuter} $\gamma$. 
  Since $B\gamma(t)=\gamma(t+\frac{2\pi r}{k})$, we can use equation \eqref{eq outerBilliardEqInvCurve}, where now $\lambda$ is constant:
  \[\gamma(t+\frac{2\pi r}{k})-\lambda\dot{\gamma}(t+\frac{2\pi r}{k})=\gamma(t)+\lambda\dot{\gamma}(t)\,,\]
  and rearrange it:
  \[\gamma(t+\frac{2\pi r}{k})-\gamma(t)=\lambda(\dot{\gamma}(t)+\dot{\gamma}(t+\frac{2\pi r}{k}))\,.\]
  Consider $\gamma(t)$ as a complex valued function, and write its Fourier expansion:
  \[\gamma(t)=\sum_{n\in\Z} c_n e^{int}\,.\]
  Plug it in the equality above:
  \[\sum_{n\in\Z} c_n(e^{i\frac{2\pi rn}{k}} - 1)e^{int} = \sum_{n\in\Z} i\lambda nc_n(e^{i\frac{2\pi r n}{k}} + 1)e^{int}\,,\]
  so we must have, for all $n\in\Z$:
  \[c_n(e^{i\frac{2\pi rn}{k}}-1)=i\lambda nc_n(e^{i\frac{2\pi rn}{k}}+1) \,.\]
  Observe that if $c_n\neq 0$, then we must have:
  \[(e^{i\frac{2\pi rn}{k}}-1)=i\lambda n(e^{i\frac{2\pi r n}{k}}+1)\,,\]
  which simplifies to the following:
  \begin{equation} \label{eq FourierConditionOuter}
	\tan\frac{\pi rn}{k}=\lambda n\,.
  \end{equation}

We make the following observation about the coefficients $c_0$,$c_1$,$c_{-1}$:
\begin{lemma}\label{lem VanishingOfFourier}
Let $\gamma:[0,2\pi]\to\C$ be a $C^2$-smooth, closed, strictly convex curve, and $B:\C\to\C$ be an $\R$-linear map of order $k\ge 3$ such that $B\gamma(t)=\gamma(t+\frac{2\pi r}{k})$, where $r\in\N$.
 Then the coefficient $c_0$ in the Fourier expansion of $\gamma$ vanishes, and at least one of the coefficients $c_1,c_{-1}$ does not vanish.
\end{lemma}  
\begin{proof}
The coefficient $c_0$ is computed by 
\[c_0=\int\limits_0^{2\pi}\gamma(t)dt\,.\]
Apply $B$ to both sides, use linearity of the integral, and periodicity of $\gamma$:
\[Bc_0=B\int\limits_0^{2\pi}\gamma(t)dt=\int\limits_0^{2\pi}B\gamma(t)dt=\int\limits_0^{2\pi}\gamma(t+\frac{2\pi r}{k})dt=c_0\,,\]
which implies that $Bc_0=c_0$.
If $c_0\neq 0$, then $1$ is an eigenvalue of $B$, but since $B$ is a linear map of order $k\ge 3$, it has no real eigenvalues.
 Therefore, $c_0 = 0$.\\
Arguing by contradiction, assume that $c_1=c_{-1}=0$. 
This means that the real and imaginary part of $\gamma$ are $L^2$ orthogonal to the subspace spanned by $\set{1,\sin t,\cos t}$.
 According to the Sturm-Hurwitz-Kellogg Theorem (see \cite{2005Pdgo}), since this space is a Chebyshev system of dimension three, the real and imaginary part of $\gamma$ must vanish at least four times.
This is impossible: since $\gamma$ is a strictly convex curve, it can intersect each axis at most twice.
This contradiction completes the proof of the lemma.
\end{proof}
  
  It follows from Lemma \ref{lem VanishingOfFourier} that equation \eqref{eq FourierConditionOuter} must be satisfied for $n=1$ or $n=-1$ (or both). In both cases we get $\lambda=\tan\frac{\pi r}{k}$. 
  Moreover, if for some $|n|>1$, we would have $c_n\neq 0$, then equation \eqref{eq FourierConditionOuter} would give us 
  \[\tan\frac{\pi rn}{k}=n\tan\frac{\pi r}{k}\,,\]
where $k \ge 3$. 
 But according to \cite[Theorem 1]{cyr2011number}, the equation
  \[\tan(nx)=n\tan(x)\] 
  cannot have solutions among rational multiples of $\pi$.
  Therefore, for all $|n|>1$ we have $c_n=0$, so $\gamma(t)$ then must be an ellipse, as its only non zero Fourier coefficients are $c_1$,$c_{-1}$. 
  This proves Theorem \ref{theorem outer}.
  \section{Symplectic Billiards - proof of Theorem \ref{theorem symplectic}} \label{section proofSymplectic}
  In Subsection \ref{subsect bg} we recalled the geometric definition and twist map structure of the Symplectic billiard map.
  Given any parameter $t$ on $\gamma$, the Symplectic billiard map $T$ can be considered as a map of a cylinder, with a vertical parameter $s$, which satisfies:
  \[T(t,s)=(t_1,s_1)\,,\]
  if and only if $s=-[\dot{\gamma}(t),\gamma(t_1)]$ and $s_1=[\gamma(t),\dot{\gamma}(t_1)]$.
  We assume that for the Symplectic billiard map there exists a rotational invariant curve, $\alpha$, of $k$-periodic orbits, of rotation number $\frac{r}{k}$, where $0<r<\frac{k}{2}$ is coprime with $k$.
  Use Lemma \ref{lem actionOfAOnGamma} to get a $2\pi$-periodic parametrization of (a homothetic copy of) $\gamma$ for which $[\gamma(t),\dot{\gamma}(t)]=1$, and an integer $m$ coprime with $k$ such that for $B=A^m$ we have $B\gamma(t)=\gamma(t+\frac{2\pi r}{k})$.
  Using this parametrization, we describe the bouncing points of the $k$-periodic Symplectic billiard orbit starting at $\gamma(t)$.
  \begin{lemma}\label{lem bouncingPointSymp}
   Let $\gamma$ be a  $C^2$-smooth, planar, strictly convex curve. 
 Assume that $\gamma$ is invariant under a linear map $A$ of order $k$, and that the Symplectic billiard map on $\gamma$ has a rotational invariant curve, $\alpha$, of $k$-periodic orbits.
 Then, the bouncing points of the $k$-periodic Symplectic billiard orbit starting at $\gamma(t)$ are $\gamma(t)$,$B\gamma(t)$,...,$B^{k-1}\gamma(t)$, where $B$ is as in Lemma \ref{lem actionOfAOnGamma}.
  \end{lemma}
  \begin{proof}
  Parametrize $\gamma$ as in Lemma \ref{lem actionOfAOnGamma}.
  Consider the map $R(t,s)=(t+\frac{2\pi r}{k},s)$ induced by $B$ on the phase space.
  We show that $T\circ R = R\circ T$.
  Let $(t_1,s_1)=T(t,s)$, meaning $s=-[\dot{\gamma}(t),\gamma(t_1)]$ and $s_1=[\gamma(t),\dot{\gamma}(t_1)]$, as mentioned above. 
  We compute:
  \[-[\dot{\gamma}(t+\frac{2\pi r}{k}),\gamma(t_1+\frac{2\pi r}{k})]=-[B\dot{\gamma}(t),B\gamma(t_1)]=-\det B [\dot{\gamma}(t),\gamma(t_1)]=s\,,\]
since $\det B=1$. Similarly:
  \[[\gamma(t+\frac{2\pi r}{k}),\dot{\gamma}(t_1+\frac{2\pi r}{k})]=s_1\,.\]
  This means that $T(t+\frac{2\pi r}{k},s)=(t_1+\frac{2\pi r}{k},s_1)$, or that $T\circ R(t,s) = R\circ T(t,s)$.
By a theorem of Birkhoff, the rotational invariant curve $\alpha$ is a graph of a function  $s=s(t)$.
  Just like in the proof of Theorem \ref{theorem euclideanK}, the fact that all points of $\alpha$ are periodic, implies that $\alpha$ is preserved by $R$, which then means that $s(t)$ is a $\frac{2\pi r}{k}$-periodic function.
  Let $S:\R\slash(2\pi\Z)\to\R\slash(2\pi\Z)$ be the function that describes the restriction of the dynamics to $\alpha$, so that:
  \[T(t,s(t))=(S(t),s(S(t)))\,,\]
  and let $\tilde{S}$ be its lift to $\R$. 
  This function is increasing, and since $\alpha$ is an invariant curve of rotation number $\frac{r}{k}$, it satisfies $\tilde{S}^k(t)=t+2\pi r$ for all~$t$. 
  Now evaluate both sides of $R\circ T = T\circ R$ on the points $(t,s(t))$:
\begin{gather*}
T\circ R(t,s(t))=T(t+\frac{2\pi r}{k},s(t))=T(t+\frac{2\pi r}{k},s(t+\frac{2\pi r}{k}))=\\
=(S(t+\frac{2\pi r}{k}),s(S(t+\frac{2\pi r}{k})))\,,
\end{gather*}
\[R\circ T(t,s(t))=R(S(t),s(S(t)))=(S(t)+\frac{2\pi r}{k},s(S(t)))\,.\]
From this computation we see that $\tilde{S}(t+\frac{2\pi r}{k})=\tilde{S}(t)+\frac{2\pi r}{k}$.
By Lemma \ref{lem impliesIsAShift}, we see that $\tilde{S}(t)=t+\frac{2\pi r}{k}$.
Thus, the bouncing points of the $k$-periodic Symplectic billiard orbit starting at $\gamma(t)$ are $\gamma(t)$,$\gamma(t+\frac{2\pi r}{k})$,...,$\gamma(t+\frac{2\pi r (k-1)}{k})$, which are just $\gamma(t)$,$B\gamma(t)$,...,$B^{k-1}\gamma(t)$.
 This proves the lemma.
  \end{proof}
Lemma \ref{lem bouncingPointSymp} implies that $\gamma(t)$, $B\gamma(t)$, $B^2\gamma(t)$ are three consecutive points for the Symplectic billiard map for all $t$. 
Therefore, using the geometric definition, there exists $\lambda(t)>0$ such that
  \begin{equation}\label{eq sympBilliardEqInvForm}
  B^2\gamma(t)-\gamma(t)=\lambda(t)B\dot{\gamma}(t)\,.
  \end{equation}
Equation \eqref{eq sympBilliardEqInvForm} was derived for the parametrization of Lemma \ref{lem actionOfAOnGamma}, but in fact, it holds for any parametrization, changing the function $\lambda(t)$ if necessary.
Moreover, if we use a parametrization for which $B\gamma(t)=\gamma(t+\frac{2\pi r}{k})$, then this equation is equivalent to 
  \begin{equation}\label{eq sympBilliardEq}
  \gamma(t+\frac{4\pi r}{k})-\gamma(t)=\lambda(t)\dot{\gamma}(t+\frac{2\pi r}{k})\,.
  \end{equation}
  As in the proof of Theorem \ref{theorem outer}, we will reparamatrize $\gamma$, so the function $\lambda(t)$ in equation \eqref{eq sympBilliardEqInvForm} will be constant, and in addition we will still have the property $B\gamma(t)=\gamma(t+\frac{2\pi r}{k})$.
  \begin{lemma}\label{lem goodParametrizationSymp}
Let $\gamma$,$A$,$m$,$B$, be as in Lemma \ref{lem actionOfAOnGamma}.
If the Symplectic billiard map on $\gamma$ has an invariant curve, $\alpha$, of $k$-periodic orbits, then there exists a $C^2$-reparametrization of $\gamma$, $\mu(s)=\gamma(f(s))$ such that:
  \begin{enumerate}
  \item $\mu(s+\frac{2\pi r}{k})=B\mu(s)$ for all $s$.
 \item The function $\lambda(s)$ determined by this parametrization in equation \eqref{eq sympBilliardEqInvForm} is constant.
 \item $\mu$ is $2\pi$-periodic.
  \end{enumerate}
  \end{lemma}
  \begin{proof}
  Parametrize $\gamma$ as in Lemma \ref{lem actionOfAOnGamma}.
   Consider the following function:
  \[G(t)=-2[B\gamma(t),\gamma(t)]\,.\]
  By definition of Symplectic billiards, the vectors inside the determinant are oriented negatively, so $G(t)>0$. 
The function $G$ is clearly $C^1$-smooth (in fact, it is even $C^2$-smooth). We show that it is $\frac{2\pi}{k}$-periodic.
Since $r$ is coprime to $k$, there exists $\l\in\N$ for which $B^\l\gamma(t)=\gamma(t+\frac{2\pi}{k})$.
As a result, when computing $G(t+\frac{2\pi}{k})$, we will have an extra factor of $\det B^\l = 1$.
  We can now use Lemma \ref{lem ODElemma} to find a constant $\lambda>0$ and a $C^2$-smooth increasing $f:\R\to\R$ for which
  \begin{equation}\label{eq ODEforSymp}
\lambda f'(s)=G(f(s))\,,
\end{equation}  
and in addition, $f(s+\frac{2\pi}{k})=f(s)+\frac{2\pi}{k}$, and thus also satisfies $f(s+\frac{2\pi r}{k})=f(s)+\frac{2\pi r}{k}$.
  Now let $\mu(s)=\gamma(f(s))$, and we show that this parametrization satisfies the required properties.
  First:
  \[\mu(s+\frac{2\pi r}{k})=\gamma(f(s+\frac{2\pi r}{k}))=\gamma(f(s)+\frac{2\pi r}{k})=B\gamma(f(s))=B\mu(s)\,,\]
  so the first item holds.
Similar computation with $B^\l$ instead of $B$ gives us $B^\l\mu(s)=\mu(s+\frac{2\pi}{k})$, and since $B^{\l k}=I$, we get that $\mu$ is $2\pi$ periodic, proving the last item.
  Now we prove the second item.
Let $\lambda(s)$ be the function which satisfies equation \eqref{eq sympBilliardEqInvForm} for this parametrization:
  \[B^2\mu(s)-\mu(s)=\lambda(s)B\dot{\mu}(s)\,.\]
  Rewrite it in terms of $\gamma$:
  \[B^2\gamma(f(s))-\gamma(f(s))=\lambda(s)f'(s)B\dot{\gamma}(f(s))\,.\]
  Recall that $[\gamma(t),\dot{\gamma}(t)]=1$, and take determinants with the vector $B\gamma(f(s))$:
  \begin{equation}\label{eq sympComputingLambda}
  \Big[B\gamma(f(s)),B^2\gamma(f(s))-\gamma(f(s))\Big]=\lambda(s)f'(s)\,.
  \end{equation}
  The linear map $B$ is of order $k\ge 3$, so it is conjugate to a rotation by angle $\frac{2\pi q}{k}$, for some $q\in\N$ coprime with $k$.
  The characteristic polynomial of $B$ is therefore $x^2-2\cos\frac{2\pi q}{k} x +1$, so by Cayley-Hamilton theorem, we get:
  \[B^2=2\cos\frac{2\pi q}{k}B-I\,.\]
  In equation \eqref{eq sympComputingLambda}, we replace $B^2$ with this expression:
  \[\lambda(s)f'(s)=\Big[B\gamma(f(s)),2\cos\frac{2\pi q}{k}B\gamma(f(s))-2\gamma(f(s))\Big]\,.\]
  By properties of determinants, we can simplify the right hand side:
  \[\lambda(s)=\frac{-2\Big[B\gamma(f(s)),\gamma(f(s))\Big]}{f'(s)}=\frac{G(f(s))}{f'(s)}\,,\]
  but by definition of $f(s)$ in equation \eqref{eq ODEforSymp}, we know that right hand side is a positive constant, proving that $\lambda(s)$ is constant.
  As a result, we conclude that $\mu(s)$ is the required reparametrization of $\gamma(t)$.
  \end{proof}
For convenience, we keep calling the curve with the parametrization constructed in Lemma \ref{lem goodParametrizationSymp} $\gamma$. 
Since in this parametrization we also have $B\gamma(t)=\gamma(t+\frac{2\pi r}{k})$, equation \eqref{eq sympBilliardEq} holds. 
View $\gamma$ as a complex valued function, and write its Fourier expansion:
\[\gamma(t)=\sum_{n\in\Z} c_ne^{int}\,.\]
Use the Fourier expansion in equation \eqref{eq sympBilliardEq}:
\[\sum_{n\in\Z}c_n(e^{i\frac{4\pi rn}{k}}-1)e^{int}=\sum_{n\in\Z}\lambda ine^{i\frac{2\pi rn}{k}}c_n e^{int}\,.\]
This means that for all $n\in\Z$ we have $c_n(e^{i\frac{4\pi rn}{k}}-1)=\lambda i ne^{i\frac{2\pi rn}{k}}c_n$.
 Therefore, if $c_n\neq 0$ then $e^{i\frac{4\pi rn}{k}}-1 = \lambda ine^{i\frac{2\pi rn}{k}}$, which is equivalent to $\lambda n=2\sin\frac{2\pi rn}{k}$.
By Lemma \ref{lem VanishingOfFourier}, we know that $c_0=0$ and one of the coefficients $c_1$ or $c_{-1}$ does not vanish.
In both cases we compute the constant $\lambda$: $\lambda=2\sin\frac{2\pi r}{k}$.
Therefore, if for some $n\in\Z$ $c_n\neq 0$, then $$\sin\frac{2\pi rn}{k}=n\sin\frac{2\pi r}{k}\,.$$
As the next elementary Lemma \ref{lem trigImpossible} shows, this is impossible, since for $k\ge 3$ (which is the assumption in Theorem \ref{theorem symplectic}), and $r$ coprime to $k$, we have $\sin\frac{2\pi r}{k}\neq0$.
\begin{lemma}\label{lem trigImpossible}
For all $n\in\Z$, $|n|\ge 2$, and $x\in \R$, we have \[|\sin(nx)|\le |n||\sin(x)|\,.\] 
Moreover, equality holds if and only if $\sin x=0$.
\end{lemma}
\begin{proof}
Since sine is an odd function, it is enough to prove for $n\ge 2$. The proof is by induction on $n$. 
For $n=2$: 
\[|\sin(2x)|=2|\sin(x)||\cos(x)|\le 2|\sin(x)|\,.\]
Moreover, we have equality if and only if $\sin(x)=0$, or if $|\cos(x)|=1$, in which case again we will have $\sin(x)=0$.
Assume the claim holds for $n$, and we show it for $n+1$:
\begin{gather*}
|\sin((n+1)x)|=|\sin(nx)\cos(x)+\cos(nx)\sin(x)|\le \\
\le|\sin(nx)||\cos(x)|+|\cos(nx)||\sin(x)|\le \\
\le|\sin(nx)|+|\sin(x)|\le n|\sin(x)|+|\sin(x)|=(n+1)|\sin(x)|\,.
\end{gather*}
Assume that $\sin(x)\neq 0$, and we show that the inequality is strict. In this case $|\cos(x)|<1$. If $\sin(nx) = 0$, then $|\cos(nx)|=1$ and:
\[|\sin((n+1)x)|=|\sin(x)||\cos(nx)|<(n+1)|\sin(x)|\,,\]
as required. If $\sin(nx)\neq 0$ then $|\cos(nx)| < 1$, so both cosines are strictly less than one, and both sines are non-zero, so the inequality from the second to the third line will be strict.
This completes the proof.
\end{proof}
As a result, we see that all Fourier coefficients of $\gamma$, except for $c_1,c_{-1}$ vanish, so $\gamma$ is necessarily an ellipse. 
This completes the proof of Theorem \ref{theorem symplectic}.

\section{Minkowski Billiards - proof of Theorem \ref{theorem Minkowski}}\label{section proofMinkowski}
In Subsection \ref{subsect bg} we recalled the definition of Minkowski billiards as a twist map. 
Given a convex body $K$, the boundary of which is $\gamma$, and a (maybe asymmetric, and $C^2$-smooth except for the origin) norm $N$, the Minkowski billiard system is the twist map for which the function $$L(t,t_1)=N(\gamma(t_1)-\gamma(t))$$ is generating.
Denoting this twist map of the cylinder by $T$, we see that it is described by $T(t,s)=(t_1,s_1)$, where \[s=dN_{\gamma(t_1)-\gamma(t)} (-\dot{\gamma}(t)),s_1=dN_{\gamma(t_1)-\gamma(t)}(\dot{\gamma}(t_1))\,.\] 
Here $dN_x$ is the differential of $N$ at the point $x$, and $\gamma$ is parametrized such that $N(\dot{\gamma}(t))=1$ for all $t$. 
In our case, we choose $N$ to be the norm induced by $K$, namely $N(x)=g_K(x)=\inf\set{r>0\mid x\in rK}$.
The norm $N$ is then $C^2$-smooth (except for the origin) since we assumed that $\gamma=\partial K$ is $C^2$-smooth.

We assume that $K$ is invariant under a linear map $A$ of order $k$. 
Assume that the Minkowski billiard map in $K$ has a rotational invariant curve, $\alpha$, of $k$-periodic orbits with rotation number $\frac{r}{k}$, where $0<r<\frac{k}{2}$ is coprime with $k$.
Let $L$ be the period of~$\gamma$. 
We use Lemma \ref{lem actionOfAOnGamma}, now using the identity $g_K(\dot{\gamma}(t))=1$ instead of the identity $[\gamma(t),\dot{\gamma}(t)]=1$ used there.
As a result, we find an integer $m$ coprime with $k$ such that for $B=A^m$, we have $B\gamma(t)=\gamma(t+\frac{Lr}{k})$.
Next, we identify the points of the $k$-periodic orbit that starts at $\gamma(t)$.
\begin{lemma}\label{lem bouncingPointsMink}
   Let $\gamma$ be a $C^2$-smooth, planar, strictly convex curve. 
 Assume that $\gamma$ is invariant under a linear map $A$ of order $k$, and that the Minkowski billiard map on $\gamma$ admits a rotational invariant curve, $\alpha$ of $k$-periodic orbits.
 Then, the bouncing points of the $k$-periodic Minkowski billiard orbit starting at $\gamma(t)$ are $\gamma(t)$,$B\gamma(t)$,...,$B^{k-1}\gamma(t)$, where $B=A^m$ for some $m\in\N$, coprime to $k$.
\end{lemma}
\begin{proof}
The linear map $B=A^m$, considered above, induces the map $R(t,s)=(t+\frac{Lr}{k},s)$ on the phase cylinder.
We show that this map commutes with $T$.
If $T(t,s)=(t_1,s_1)$, then we need to show that $T(t+\frac{Lr}{k},s)=(t_1+\frac{Lr}{k},s_1)$. 
Indeed: 
\begin{gather*}
d(g_K)_{\gamma(t_1+\frac{Lr}{k})-\gamma(t+\frac{Lr}{k})}(-\dot{\gamma}(t+\frac{Lr}{k}))=d(g_K)_{B\gamma(t_1)-B\gamma(t)}(-B\dot{\gamma}(t))=\\
=d(g_K)_{\gamma(t_1)-\gamma(t)}(-\dot{\gamma}(t))=s\,.
\end{gather*}
The last equality follows from the fact that $g_K(Bx)=g_K(x)$ for all $x$, which implies that $d(g_K)_{Bx}\circ B = d(g_K)_x$.
Similar computation gives that
\[d(g_K)_{\gamma(t_1+\frac{Lr}{k})-\gamma(t+\frac{Lr}{k})}(\dot{\gamma}(t_1+\frac{Lr}{k}))=d(g_K)_{\gamma(t_1)-\gamma(t)}(\dot{\gamma}(t_1))=s_1\,.\]
This proves that $T$ and $R$ commute. 

By Birkhoff's theorem, the rotational invariant curve $\alpha$ is the graph of a function $s=s(t)$, and there exists a function $S:\R\slash(L\Z)\to\R\slash(L\Z)$ which describes the dynamics on $\alpha$: $T(t,s(t))=(S(t),s(S(t)))$.
Since the rotation number of $\alpha$ is $\frac{r}{k}$, this circle map lifts to an increasing function $\tilde{S}:\R\to\R$ which satisfies $\tilde{S}^k(t)=t+Lr$.
As in the proof of Theorem \ref{theorem euclideanK} in Section \ref{section birkhoff}, the fact that every point of $\alpha$ is periodic implies that $s$ is $\frac{Lr}{k}$-periodic function.
Evaluate both sides of $T\circ R= R\circ T$ on the points $(t,s(t))$:
\begin{gather*}T\circ R(t,s(t))= T(t+\frac{Lr}{k},s(t))=T(t+\frac{Lr}{k},s(t+\frac{Lr}{k}))=\\
=(S(t+\frac{Lr}{k}),s(S(t+\frac{Lr}{k})))\,,
\end{gather*}
\begin{gather*}
R\circ T(t,s(t))=R(S(t),s(S(t)))=(S(t)+\frac{Lr}{k},s(S(t)))\,.
\end{gather*}
As a result, we obtain $\tilde{S}(t+\frac{Lr}{k})=\tilde{S}(t)+\frac{Lr}{k}$, and together with the fact that $\tilde{S}^k(t)=t+Lr$, using Lemma \ref{lem impliesIsAShift}, we get that $\tilde{S}(t)=t+\frac{Lr}{k}$.
Consequently, the bounce points of the $k$-periodic Minkowski billiard orbit starting at $\gamma(t)$ are $\gamma(t)$,$\gamma(t+\frac{Lr}{k})$,...,$\gamma(t+(k-1)\frac{Lr}{k})$ which are just $\gamma(t)$,$B\gamma(t)$,...,$B^{k-1}\gamma(t)$.
\end{proof}
Consider the action functional on this invariant curve:
\[F(t)=\sum\limits_{i=1}^k g_K(B^i\gamma(t)-B^{i-1}\gamma(t))\,.\]
Since $\alpha$ is an invariant circle, by Mather theory, the function $F$ must be constant (see, e.g., \cite[Theorem 35.2]{gole2001symplectic}).
Using the fact that $g_K\circ B = g_K$, we see that the function $g_K(B\gamma(t)-\gamma(t))$ is therefore also constant.
\begin{lemma}\label{lem equivToMoreSymmetry}
Let $K$ be a strictly convex body with $C^2$-smooth boundary invariant under a linear map $A$ of order $k$, and $B=A^m$ with $m$ coprime to $k$.
Then $g_K(B\gamma(t)-\gamma(t))$ is constant if and only if $K$ is invariant under a linear map of order $ak$ that commutes with $A$, where 
\[\begin{cases}
a = 1\,, \textrm{if } k\equiv 2 \pmod{4} \,,\\
a = 2\,, \textrm{if } k\equiv 0 \pmod{4} \,,\\
a = 4\,, \textrm{if } k\equiv 1 \pmod{2} \,.\\
\end{cases}\]
\end{lemma}
\begin{proof}
The fact that $g_K(B\gamma(t)-\gamma(t))$ is constant is equivalent to the following: there exists a constant $\lambda>0$ and a function $H(t)$ such that for all $t$ (see Figure \ref{fig MinkowskiGKCondition}):
\begin{equation}\label{eq gKConstCond}
B\gamma(t)-\gamma(t)=\lambda\gamma(H(t))\,.
\end{equation}
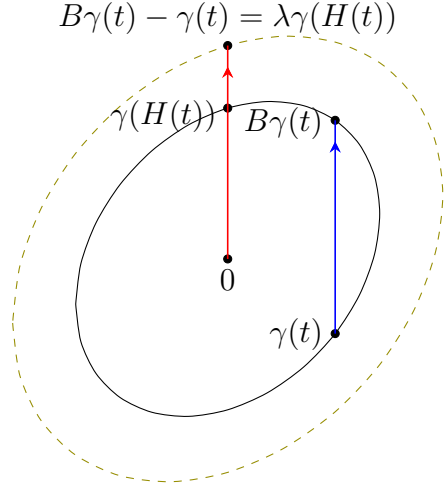
\begin{figure}
\begin{center}
\begin{tikzpicture}[scale = 2]
\tikzset{
arr/.style={postaction=decorate,
decoration={markings,
mark=at position .9 with {\arrow[thick]{#1}}
}}}
  \draw[domain = -180:180, smooth, variable = \t, black] plot({cos(\t)},{0.3*cos(\t)+sin(\t)});
    \draw[domain = -180:180, smooth, variable = \t, olive,dashed] plot({1.4148*cos(\t)},{1.4148*(0.3*cos(\t)+sin(\t))});
    \tkzDefPoint(0,0){O};
    \tkzDefPoint(0.70710,-0.49497){A};
\tkzDefPoint(0.70710,0.91923){B};
\tkzDefPoint(0,1.4148){C};
\tkzDefPoint(0,1){D};
\tkzDrawPoints(O,A,B,C,D);
\tkzDrawSegment[arr=stealth,blue](A,B);
\tkzDrawSegment[arr=stealth,red](O,C);
\node[left] at (0.7,-0.5) {$\gamma(t)$};
\node[left] at (0.7,.9) {$B\gamma(t)$};
\node[left] at (0,0.95) {$\gamma(H(t))$};
\node[above] at(0,1.42) {$B\gamma(t)-\gamma(t)=\lambda \gamma(H(t))$};
\node[below] at (0,0) {$0$};
\end{tikzpicture}
\caption{The condition $g_K(B\gamma(t)-\gamma(t))=\textrm{const}$.}\label{fig MinkowskiGKCondition}
\end{center}
\end{figure}
Since $A$ is a linear map of order $k$, it is conjugated to a rotation that generates the group of rotations by $\frac{2\pi}{k}$. 
Thus, we can change basis to assume that $A$ is a rotation, and then $B$ would be a rotation by angle $\frac{2\pi\l}{k}$, with $\l$ coprime to $k$, $R_{\frac{2\pi\l}{k}}$.
We reparametrize $\gamma$ by $\psi$, the angle that the outer normal makes with the $x$ axis of this basis.
In this parametrization, we have $B\gamma(\psi)=\gamma(\psi+\frac{2\pi\l}{k})$.
Differentiate equation \eqref{eq gKConstCond} with respect to $\psi$:
\begin{equation}\label{eq MinkPreDirComparing}
(B-I)\dot{\gamma}(\psi)=\lambda H'(\psi)\dot{\gamma}(H(\psi))\,.
\end{equation}
It can be checked that if $B=R_{\frac{2\pi\l}{k}}$ then $B-I=2\sin\frac{2\pi\l}{k}\ R_{\frac{\pi}{2}+\frac{\pi\l}{k}}$.
As a result, the curve $(B-I)\gamma(\psi)$ has the same orientation as $\gamma$.
Hence, by equation \eqref{eq gKConstCond}, $\gamma(H(\psi))$ also has the same orientation as $\gamma$, so $H$ is increasing.
Thus, if we compare the directions of the vectors in both sides of equation \eqref{eq MinkPreDirComparing}, we see that:
\[(\psi+\frac{\pi}{2})+(\frac{\pi}{2}+\frac{\pi\l}{k})=H(\psi)+\frac{\pi}{2}\,,\]
from which it follows that:
 \[H(\psi)=\psi+\frac{\pi\l}{k}+\frac{\pi}{2}\,.\]
As a result, we get the equality
\begin{equation}\label{eq MinkowskiCondition}
(B-I)\gamma(\psi)=\lambda\gamma(\psi+\frac{\pi\l}{k}+\frac{\pi}{2})\,.
\end{equation}

So we see that the condition that $g_K(B\gamma(t)-\gamma(t))$ is constant is equivalent to equation \eqref{eq MinkowskiCondition}. 
We now identify the constant $\lambda$.
We view $\gamma$ as a complex valued function, and respectively identify $B$ with $e^{i\frac{2\pi\l}{k}}$.
 Write the Fourier expansion of $\gamma$: $\gamma(\psi)=\sum\limits_{n\in\Z}c_n e^{in\psi}$.
Substitute it in equation \eqref{eq MinkowskiCondition}, and compare the corresponding Fourier coefficients, to get that for all $n\in\Z$, we have:
\[c_n (e^{i\frac{2\pi\l}{k}}-1) = \lambda c_n e^{in(\frac{\pi\l}{k}+\frac{\pi}{2})}\,.\]
One can easily see from this equation that $c_{-1}$ must vanish. Hence,
by Lemma \ref{lem VanishingOfFourier},  $c_1\neq 0$. 
Therefore, we have:
\[e^{i\frac{2\pi\l}{k}}-1=\lambda ie^{i\frac{\pi\l}{k}}\,.\]
From which we get that $\lambda=2\sin\frac{\pi\l}{k}$.
However, since $$B-I=2\sin\frac{\pi\l}{k}\ R_{\frac{\pi}{2}+\frac{\pi\l}{k}}=\lambda R_{\frac{\pi}{2}+\frac{\pi\l}{k}},$$ equation \eqref{eq MinkowskiCondition} implies that $\gamma$ is invariant under $R_{\frac{\pi}{2}+\frac{\pi\l}{k}}$.
We conclude that this condition is equivalent to the fact that $g_K(B\gamma(t)-\gamma(t))$ is constant. 

Recall, that by assumption, $\gamma$ is invariant under $R_{\frac{2\pi}{k}}$.
As a result, $\gamma$ is invariant under rotations of angles 
\[\frac{2\pi}{k}p+(\frac{\pi}{2}+\frac{\pi\l}{k})q=\frac{\pi}{2k}(4p+(k+2\l)q)\,,\]
 for all $p,q\in\Z$.
 This means that $\gamma$ is also invariant under a rotation by $\frac{\pi}{2k}\gcd(4,k+2\l)$.
 If $k$ is odd, then $k+2\l$ is odd, so $\gcd(4,k+2\l)=1$, so $\gamma$ is invariant under $R_{\frac{\pi}{2k}}$ which is an element of order $4k$ in $\textrm{GL}(2,\R)$.
 If $k$ is even, then $\l$ must be odd, since $k,l$ are coprime, and hence $2\l$ has remainder $2$ modulo $4$. 
If $k\equiv 0(\textrm{mod }4)$ then $\gcd(4,k+2\l)=2$ so $\gamma$ is invariant under $R_{\frac{\pi}{k}}$, which is of order $2k$.
 Finally, if $k\equiv 2(\textrm{mod }4)$ then $\gcd(4,k+2\l)=4$, so $\gamma$ is invariant under $R_{\frac{2\pi}{k}}$, which is of order $k$.
In any case we see that $\gamma$ is invariant under an element of order $ak$, where $a$ is as in the formulation of the theorem. 
Since any two rotations (with a common center) in $\R^2$ commute, this map of order $ak$ commutes with $A$.

Conversely, suppose that $\gamma$ is invariant under a linear map of order $ak$.
 Choosing a suitable basis, we may assume that this linear map is $R_{\frac{2\pi \l}{ak}}$, with $\l$ being coprime to $ak$.
This means that $\gamma$ is also invariant under $R_{\frac{2\pi}{ak}}$, and therefore, it is also invariant under $R_{\frac{2\pi}{k}}$ and $R_{\frac{\pi}{k}+\frac{\pi}{2}}$.
In this basis, we choose $A=R_{\frac{2\pi}{k}}$, and $B=R_{\frac{2\pi m}{k}}=A^m$ with $m$ coprime with $k$.
The curve $\gamma$ is invariant under $A$ and under $R_{\frac{\pi}{2}+\frac{\pi}{k}}$.
As we computed before, we have $B-I=2\sin\frac{\pi m}{k}R_{\frac{\pi}{2}+\frac{\pi m}{k}}$.
In order to prove that $g_K(B\gamma(t)-\gamma(t))$ is constant,  we need to show, by equation \eqref{eq MinkowskiCondition}, that $\gamma$ is invariant under $R_{\frac{\pi}{2}+\frac{\pi m}{k}}$.
Indeed, $\gamma$ is invariant under $R_{\frac{\pi}{2}+\frac{\pi}{k}}$, so we only need to show that it is also invariant under $R_{\frac{\pi (m-1)}{k}}$.
If $k\not\equiv 2\pmod{4}$ this is apparent since in this case $\gamma$ is invariant under $R_{\frac{\pi}{k}}$. 
If $k\equiv 2\pmod{4}$, then $m$ must be odd, since $k$ is even. 
Consequently, $\frac{\pi(m-1)}{k}$ is an integer multiple of $\frac{2\pi}{k}$, and $\gamma$ is invariant under $R_{\frac{2\pi}{k}}$, and in this case again $\gamma$ is invariant under $R_{\frac{\pi(m-1)}{k}}$.
This completes the proof of the lemma.
\end{proof}
To finish the proof of Theorem \ref{theorem Minkowski}, we will prove that this condition is sufficient.
 Suppose, that $\gamma$ is a curve which is invariant under a linear map $A$ of order $k$, and for which $g_K(B\gamma(t)-\gamma(t))$ is constant, where $B=A^m$ and $m$ is coprime with $k$.
 We need to prove that the Minkowski billiard map in $K$ has rotational invariant curves of $k$-periodic orbits of all rotation numbers $\frac{r}{k}$, where $r$ is coprime with $k$. 
Indeed, since $\gamma$ is invariant under $A$, for all $t$, $B\gamma(t)=A^m\gamma(t)$ is also a point on $\gamma$. 
Thus, we need to show that the points $\gamma(t)$,$B\gamma(t)$,...,$B^{k-1}\gamma(t)$ form a Minkowski billiard orbit. 
It is enough to verify, that for all $t$, $\tau=t$ is a critical point of the function 
\[f(\tau)=g_K(B\gamma(t)-\gamma(\tau))+g_K(\gamma(\tau)-B\inv\gamma(t))\,,\]
where $\gamma$ is parametrized such that $g_K(\dot{\gamma}(t))=1$.
We compute $f'(t)$:
\[f'(t)=d(g_K)_{B\gamma(t)-\gamma(t)}(-\dot{\gamma}(t))+d(g_K)_{\gamma(t)-B\inv\gamma(t)}(\dot{\gamma}(t))\,.\]
Using the fact that $d(g_K)_{Bx}\circ B = d(g_K)_x$,  the second summand is equal to $d(g_K)_{B\gamma(t)-\gamma(t)}(B\dot{\gamma}(t))$.
As a result, we obtain
\[f'(t)=d(g_K)_{B\gamma(t)-\gamma(t)}(B\dot{\gamma}(t)-\dot{\gamma}(t))=0\,,\]
since this is just the derivative of the constant function $g_K(B\gamma(t)-\gamma(t))$.
Now we explain why the rotation number of this orbit can be any number of the form $\frac{r}{k}$ with $r$ coprime to $k$.
In a suitable basis, $A=R_{\frac{2\pi\l}{k}}$, with $\l$ coprime to $k$. 
The number $m$ is also coprime to $k$, so $B=R_{\frac{2\pi\l'}{k}}$, with $\l'$ coprime to $k$.
We can choose $m$ to have $\l'=r$.
If we parametrize $\gamma$ by the angle $\psi$ of the outer normal, then the points of the $k$-periodic orbits will be $\gamma(\psi)$,$\gamma(\psi+\frac{2\pi r}{k})$,...,$\gamma(\psi+(k-1)\frac{2\pi r}{k})$, and this orbit has rotation number $\frac{r}{k}$.
\section{Open Questions}\label{section followUp}

To conclude, we ask several questions that arise naturally:
\begin{enumerate}

\item Can one remove symmetry assumption in Theorem \ref{thm BirkhoffForMinkowski}?
Saying differently, does total integrability of Minkowski billiard in $K$ with the norm determined by $K$ implies that $K$ is an ellipse? Proof of this fact may require new integral-geometric tools for Minkowski billiards.

\item Minkowski billiards in a $k$-symmetric domain $K$ with the usual $L^2$ norm correspond to the Birkhoff billiard system, for which we proved rigidity.
But as we showed, there are many examples of $k$-symmetric convex domains $K$ for which the Minkowski billiard map in $K$ with the norm induced by $K$ admits an invariant curve of $k$ periodic orbits.
Are there other norms which provide examples to each of these types of behaviors?

\item Does Theorem \ref{theorem euclideanK} for Birkhoff billiards generalize to higher dimensions?
 In this case, the billiard ball map is not a a self map of a cylinder, but of a ball bundle over a sphere.
\item Can our results be generalized for Birkhoff and Outer billiard systems in constant curvature surfaces? To answer this question one needs to understand  the analogs of the equation \eqref{gutkin} used in the proof of Theorems \ref{theorem euclideanK},\ref{theorem outer} in the case of constant curvature surfaces.

\item What can be said about other dynamical systems, like the Wire billiards, introduced in \cite{bialy2019wire}?
This is also a twist maps of a cylinder, so one can expect that the ideas used in the proof of Theorems 1-4 might work for this system as well.
\end{enumerate}
\bibliography{rotationBib}
\bibliographystyle{abbrv}
\end{document}